\documentclass[12pt]{article}
\usepackage{tikz}
\usepackage{amsmath}
\usepackage{indentfirst, latexsym, bm, amssymb}

\begin{document}

\newtheorem{theorem}{Theorem}[section]
\newtheorem{corollary}[theorem]{Corollary}
\newtheorem{definition}[theorem]{Definition}
\newtheorem{conjecture}[theorem]{Conjecture}
\newtheorem{question}[theorem]{Question}
\newtheorem{problem}[theorem]{Problem}
\newtheorem{lemma}[theorem]{Lemma}
\newtheorem{proposition}[theorem]{Proposition}
\newtheorem{example}[theorem]{Example}
\newenvironment{proof}{\noindent {\bf
Proof.}}{\rule{3mm}{3mm}\par\medskip}
\newcommand{\remark}{\medskip\par\noindent {\bf Remark.~~}}
\newcommand{\pp}{{\it p.}}
\newcommand{\de}{\em}

\newcommand{\JEC}{{\it Europ. J. Combinatorics},  }
\newcommand{\JCTB}{{\it J. Combin. Theory Ser. B.}, }
\newcommand{\JCT}{{\it J. Combin. Theory}, }
\newcommand{\JGT}{{\it J. Graph Theory}, }
\newcommand{\ComHung}{{\it Combinatorica}, }
\newcommand{\DM}{{\it Discrete Math.}, }
\newcommand{\ARS}{{\it Ars Combin.}, }
\newcommand{\SIAMDM}{{\it SIAM J. Discrete Math.}, }
\newcommand{\SIAMADM}{{\it SIAM J. Algebraic Discrete Methods}, }
\newcommand{\SIAMC}{{\it SIAM J. Comput.}, }
\newcommand{\ConAMS}{{\it Contemp. Math. AMS}, }
\newcommand{\TransAMS}{{\it Trans. Amer. Math. Soc.}, }
\newcommand{\AnDM}{{\it Ann. Discrete Math.}, }
\newcommand{\NBS}{{\it J. Res. Nat. Bur. Standards} {\rm B}, }
\newcommand{\ConNum}{{\it Congr. Numer.}, }
\newcommand{\CJM}{{\it Canad. J. Math.}, }
\newcommand{\JLMS}{{\it J. London Math. Soc.}, }
\newcommand{\PLMS}{{\it Proc. London Math. Soc.}, }
\newcommand{\PAMS}{{\it Proc. Amer. Math. Soc.}, }
\newcommand{\JCMCC}{{\it J. Combin. Math. Combin. Comput.}, }
\newcommand{\GC}{{\it Graphs Combin.}, }

\title{ Laplacian Coefficient, Matching Polynomial  and Incidence Energy of  of Trees with Described Maximum Degree
\thanks{
This work is supported by National Natural Science Foundation of
China (Nos. 11271256 and  11531001), The Joint Israel-China Program (No.11561141001), Innovation Program of Shanghai Municipal Education Commission (No.14ZZ016), The Ph.D. Programs Foundation of Ministry of Education of China (No.20130073110075).
}
}
\author{
Ya-Lei Jin$^{a}$, Yeong-Nan Yeh$^{b}$,  Xiao-Dong Zhang$^{a}$\thanks{Corresponding  author ({\it E-mail address:}
xiaodong@sjtu.edu.cn)
}
\\
{\small $^a$Department of Mathematics, MOE-LSC, and SHL-MAC}\\
{\small Shanghai Jiao Tong University} \\
{\small  800 Dongchuan road, Shanghai, 200240,  China}\\
{\small $^b$Institute of Mathematics,}\\
{\small  Academia Sinica, Taibei 11529, Taiwan,}\\
 }
\date{}
\maketitle
 \begin{abstract}
 Let $\mathcal{L}(T,\lambda)=\sum_{k=0}^n
(-1)^{k}c_{k}(T)\lambda^{n-k}$ be the characteristic polynomial of its Laplacian
matrix of a tree $T$. This paper studied some properties  of the generating function of the coefficients sequence $(c_0, \cdots, c_n)$  which are related with  the matching polynomials of division tree of $T$. These results, in turn, are used to characterize all extremal trees having the minimum Laplacian coefficient
generation function and
the minimum incidence energy of trees with described maximum degree, respectively.

 \end{abstract}


\vskip 0.5cm

\section{Introduction}
Let $G=(V(G), E(G))$ be a simple  graph with vertex set $V(G)=\{v_1, \cdots, v_n\}$ and edge set $E(G)$. Let $A(G)=(a_{ij})$  and $D(G)=(d(v_1), \cdots, d(v_n))$ be
its adjacency and degree diagonal matrices, respectively. Then the {\it  Laplacian matrix } of $G$ is defined to be $L(G)=D(G)-A(G)$. {\it The Laplacian polynomial} $\mathcal{L}(G,\lambda)$ of $G$ is the characteristic
polynomial of its Laplacian matrix $L(G)$, i.e.,
\begin{equation}\mathcal{L}(G,\lambda)=det(\lambda I_{n}-L(G))=\sum_{k=0}^n
(-1)^{k}c_{k}(G)\lambda^{n-k}.\end{equation}
It is well known that
$c_{0}(G)=1,c_{n}(G)=0,c_{1}(G)=2|E(G)|$ and $ c_{n-1}=n\tau(G)$, where $\tau(G)$ is the number of the spanning trees.
In addition,

\begin{equation}\varphi(T,x)=c_0+c_1x+\cdots c_{n-1}x^{n-1}\end{equation}
 is called the {\it Laplacian coefficient generation function} of $T$.
  Mohar \cite{Mohar2007}  proposed a new notation of  poset consisting of all trees with Laplacian coefficients.  Let $(\mathcal{T}_n, \preceq)$  be a poset consisting of all trees of order $n$  with $\preceq$, where   $T_1\preceq T_2$, if $(c_0(T_1), \cdots, c_{n-1}(T_1))\le
(c_0(T_2), \cdots, c_{n-1}(T_2))$, i.e.,  $c_i(T_1)\le c_i(T_2)$ for $i=0, \ldots, n-1$. Moreover, write $T_1\prec T_2$ if $T_1\preceq T_2$ and there exists a $k$ with $c_k(T_1)<c_k(T_2)$. Further, he established the monotone relations under two graph operations,  which presents a strengthening of Zhou and Gutman's result \cite{zhou 2008} that  $(\mathcal{T}_n, \preceq)$ has a unique maximal element the path $P_n$ and a unique minimal element the star $K_{1, n-1}$. Besides, he  \cite{Mohar2007} also proposed some problems on how to order trees with the Laplacian coefficients.  In addition, Il\'{i}c\cite{Ilic2010} determined the extremal tree which
has minimal Laplacian coefficients in all $n-$vertex trees with a fixed matching number. Stevanov\'{i}c and Il\'{i}c \cite{Stevanovic2009b} characterized the minimum and maximum elements in the poset of unicyclic graphs of order $n$ with $\preceq$.
Tan \cite{Tan2011} proved that the poset of unicyclic graphs of order $n$ and fixed
matching number with $\preceq$ has only one minimal element. The study on the Laplacian coefficients has attracted more and more attention.  The readers are referred to 
 \cite{heuberger2008}, \cite{heuberger2009a},
  \cite{Zhang2009} and references therein.

Let $I(G)$ be  the vertex-edge
incidence matrix, i.e., an $(n\times m)$-matrix whose $(i, j)$-entry is 1 if the vertex
$v_i$ is incident to the edge $e_j$ , and  0 otherwise. Then {\it incidence
energy $IE(G)$ } (see  \cite{nikiforov2007},  \cite{gutman2009} or \cite{jooyandeh2009}) of $G$ is defined to be the sum of the singular values of $I(G)$, i.e., the sum of the square roots of all eigenvalues of  $I(G)I(G)^T$.
 On the other hand, the extremal trees with the  minimal Wiener index of trees with  maximum degree $\Delta$ has attracted considerable attention. Liu et al. \cite{liu2000}, Fischermann et
al. \cite{fischermann2002} and Jelen et al. \cite{jelen2003}
independently determined all trees which have the minimum Wiener
indices  among all trees of order $n$ and maximum degree $\Delta$ by different approaches.
 Kirk and Wang \cite{kirk2008} studied the number of subtrees of a tree with  the maximum degree. Zhang \cite{zhang2008}  characterized the extremal tree with the maximum  Laplacian spectral radius among all trees of order $n$ with the maximum degree.  These results  motivate  us to consider the following problem in this
paper.
\begin{problem}
Characterize all minimal elements in the poset $(\mathcal{T}_{n, d+1}, \preceq)$, where $\mathcal{T}_{n, d+1}$ is the set of all trees with the  maximum degree $d+1$.
\end{problem}

In order to analyze this problem, some more notations are introduced.  {\it A rooted $d-$ary tree} is a rooted tree of which every vertex has $0$ or $d$ children. The (rooted) complete $d-$ary tree of height $h-1$, denoted by $C_h$, is a rooted $d-$ary tree such that the height of each pendent vertex  is $h-1$. Then $C_1$ consists of a single vertex and the root of $C_h$ has $d$ branches which are $C_{h-1}$. Moreover, the degree of the root in rooted complete $d-$ary tree $C_h$ is $d$.
\begin{definition}
A  rooted tree $T$ with the root $v_0$ and the maximum degree $d+1$ is called {\it $(d+1)$-greedy tree,} denoted by $T_{d+1}^*$,  if the following properties have been satisfied:

(1) the degree of $v_0$  is $d+1$, i.e., $deg(v_0)=d+1$.

(2) The height of any two pendent vertices of $T$ differs by at most 1, where the height of a vertex $v$ in $T$ is equal to the distance between $v$ and $v_0$.

(3) For any  vertex $v$ in $T,$  there is at most one $T(u)$ is incomplete $d-$ary tree, where $u$ is the children of $v$ and $T(u)$ is the rooted subtree of $T$ that is induced by $u$ and all of its successors in $T$, the root of which is $u$.
\end{definition}

Let $T=(V(T), E(T))$ be a tree with $V(T)=\{v_1, \cdots, v_n\}$ and $E(T)=\{e_1, \cdots, e_{n-1}\}$. The {\it subdivision tree } of $T$ is defined to be a tree $S(T)=(V(S(T)), E(S(T)))$ with vertex set $V(S(T))=V(T)\bigcup E(T)$,
and $v_i$ and $e_j$ are adjacency in $S(T)$ if and only if $v_i$ is incidence with $e_j$ in $T$. In other words, $S(T)$ is the tree obtained from $T$ by inserting a new vertex in each edge in $T$.

The main results of this paper can be stated as follows.
\begin{theorem}\label{main1}
$T_{d+1}^*$ is the unique tree  with the minimum  Laplacian coefficient generation function in  $\mathcal{T}_{n, d+1}$, i.e. for any tree $T\in \mathcal{T}_{n, d+1}$ and $x>0$,
$$\varphi(T_{d+1}^*,x)\le \varphi(T,x)$$
with equality if and only if $T=T_{d+1}^*$.
\end{theorem}
\begin{theorem}\label{main2}
$T_{d+1}^*$ is the unique tree with the minimum incidence energy in  $\mathcal{T}_{n, d+1}$, i.e.,  for any tree $T\in \mathcal{T}_{n, d+1}$ ,
$$IE(T_{d+1}^*)\le IE(T)$$
with equality if and only if $T=T_{d+1}^*$.
\end{theorem}

  The approach to the proof of Theorem~\ref{main1} is different from some known technique, although the extremal trees for different graph variants such as the Wiener index, the Laplacian spectral radius, the number of subtrees among all trees of order $n$ and the maximum degree $\Delta$ are greedy trees. The The rest of this paper is organized as follows. In Section 2, some preliminary on the matching polynomials of tree are presented. In Section 3, the proof of Theorem~\ref{main1} is presented. In Section 4,  we give the proof of Theorem~\ref{main2} and propose a conjecture.

\section{Matching generating function}

 For a  tree $T$, let $m(T,k)$ be the number of matchings of $T$ containing precisely $k$ edges
(shortly $k-$matchings). It is convenient to define $m(T, 0)=1$. Then the {\it matching generating function} of $T$ is defined to

 \begin{equation}M(T,x)=\sum_{k\ge 0}m(T,k)x^k.\end{equation}
  If $T$ is a rooted tree, let $m_1(T,k)$ be the number of $k-$matchings of $T$ saturating the root and $m_0(T,k)$ be the number of $k-$matchings of $T$ not saturating the root.
 Denote by
\begin{equation}M_i(T,x)=\sum_{k\ge 0}m_i(T,k)x^k,~\mbox{for $i\in\{0,1\}$}.
\end{equation}
and
 \begin{equation} \tau(T,x)=\frac{M_0(T,x)}{M(T,x)}.\end{equation}
Clearly $M(T,x)=M_0(T,x)+M_1(T,x)$. It follows from \cite{heuberger2009b} that
\begin{lemma}\cite{heuberger2009b} Let $T$ be a rooted tree with root $v$ and the branches
$T_1,\cdots, T_k$. Then
\begin{equation}\label{HW1}
M_0(T,x)=\prod_{j=1}^kM(T_j,x),
\end{equation}
\begin{equation}\label{HW2}
M_1(T,x)=x\sum_{j=1}^k\frac{M_0(T_j,x)}{M(T_j,x)}\prod_{i=1}^kM(T_i),
\end{equation}

\begin{equation}\label{HW3}
\tau(T,x)=\frac{1}{1+x\sum_{j=1}^k\tau(T_j,x)}.
\end{equation}
\end{lemma}

\begin{lemma}\label{identify}Let $T_1$ and $ T_2$ be two vertex disjoint trees with roots $u$ and $v$, respectively. If $T$  is the tree obtained from $T_1$ and $T_2$ by identifying $u$ and $v$, then $$M(T,x)=M(T_1,x)M_0(T_2,x)+M_0(T_1,x)M_1(T_2,x).$$
\end{lemma}
\begin{proof}
It is easy to see that $$m(T,k)=\sum_{i=0}^k\left(m(T_1,i)m_0(T_2,k-i)+m_0(T_1,i)m_1(T_2,k-i)\right).$$
Hence
\begin{eqnarray*}
M(T,x)
&=&\sum_{k\ge 0}\sum_{i=0}^k(m(T_1,i)m_0(T_2,k-i)+m_0(T_1,i)m_1(T_2,k-i))x^k\\
&=&\sum_{k\ge 0}\sum_{i=0}^k(m(T_1,i)m_0(T_2,k-i)x^k)+\sum_{k\ge 0}\sum_{i=0}^k(m_0(T_1,i)m_1(T_2,k-i)x^k)\\
&=&M(T_1,x)M_0(T_2,x)+M_0(T_1,x)M_1(T_2,x).
\end{eqnarray*}
This completes the proof.
\end{proof}

\begin{lemma}\label{subdivision-rec}
 Let $T$ be a rooted tree with root $v$ and the branches
$T_1,\cdots, T_k$. If $S(T)$, $S(T_1), \cdots, S(T_k)$ are the subdivision trees of $T, T_1, \cdots, T_k$, respectively, then
\begin{equation}\label{LM-z0}
M_0(S(T),x)=\prod_{j=1}^k(xM_0(S(T_j),x)+M(S(T_j),x)),
\end{equation}
\begin{equation}\label{LM-z1}
M_1(S(T),x)=x\sum_{j=1}^k\frac{M(S(T_j),x)}{xM_0(S(T_j),x)+M(S(T_j),x) }M_0(S(T),x),
\end{equation}
\begin{equation}\label{LM-tau}
\tau(S(T),x)=\frac{1}{1+\sum_{j=1}^k\frac{x}{1+x\tau(S(T_j),x)}}.
\end{equation}
\end{lemma}
\begin{proof}
Let $v$, $v_1, \ldots, v_k$ be the roots of $T, T_1, \ldots, T_k$, respectively and let  $u_1, \cdots, u_k$  be new vertices in the edges $vv_1, \cdots, vv_k$ in $S(T)$, respectively.
By (\ref{HW1}) and (\ref{HW2}),
$$M_0(S(T),x)=\prod_{j=1}^kM(S(T_j)+v_ju_j,x)=\prod_{j=1}^k\left[xM_0(S(T_j),x)+M(S(T_j),x)
\right].$$
On the other hand, by (\ref{HW1}) and (\ref{HW2}),
\begin{eqnarray*}
M_1(S(T),x)&=&x\sum_{j=1}^k\frac{M_0(S(T_j)+v_ju_j,x)}{M(S(T_j)+v_ju_j,x)) }\prod_{i=1}^kM(S(T_i)+v_ju_j,x)\\
&=& x\sum_{j=1}^k\frac{M(S(T_j),x)}{xM_0(S(T_j),x)+M(S(T_j),x) }M_0(S(T),x)
\end{eqnarray*}
Hence it follows from (\ref{LM-z0}) and (\ref{LM-z1}) that
$$\tau(S(T),x)=\frac{M_0(S(T),x)}{M(S(T),x)}=\frac{1}{1+\sum_{j=1}^k\frac{x}{1+x\tau(S(T_j),x)}}.$$
Hence the assertions hold.\end{proof}

\begin{lemma}\label{subgraph}
Let $T$ be a tree with root $v$ and the branches $T_1,...,T_k$. If $T'$ is a proper subtree of $T$ with root $v$ and the branches $T_1',...,T_r'$, then $\tau(S(T'),x)>\tau(S(T),x)$.
\end{lemma}
\begin{proof}
Denote $|V(T')|=n'$.  If $n'=1$, then  $\tau(S(T'),x)=1$ and the assertion holds immediately. Assume that the assertion holds for $n'<t$. Since $T'$ is a proper subgraph of $T$ with root $v$ and the branches $T_1',...,T_r'$, it is easy to see that, without loss of generality, $T_i',~1\le i\le r$, are subtrees of $T_1,...,T_r$, respectively. Moreover, there is at least one $i$ such that $T_i'$ is a proper subtree of $T_i$ or $k>r$.  By $(\ref{HW3})$, we have
\begin{eqnarray*}
\tau(S(T),x)=\frac{1}{1+\sum_{i=1}^k\frac{x}{1+x\tau(S(T_i),x)}}, \ \
\tau(S(T'),x)=\frac{1}{1+\sum_{i=1}^r\frac{x}{1+x\tau(S(T_i'),x)}}.
\end{eqnarray*}
By the induction hypothesis, $\tau(S(T_i'),x)\ge \tau(S(T_i),x),~1\le i\le r$ with at least one  strict inequality  or $k>r$.
Hence, $\tau(S(T'),x)>\tau(S(T),x)$.
\end{proof}
Now we are ready to prove the following exchange theorem which plays a key role in this paper.
\begin{theorem}\label{exchange0}
Let $T_1$ be a  tree with root $v_1$ and  the branches $L_{1}, \ldots, L_{d_1}$ and $T_2$ be a rooted tree with root $v_2$ and  the branches $R_{1}, \ldots, R_{d_2}$. Let $T_0$ be any tree with two  vertices $u$ and $v$.
Let $T$ be a tree obtained  from $T_1, T_2, T_0$ by identifying $v_1$ and $u$, $v_2$ and $v$, respectively. Let $d\ge \max\{d_1, d_2\}$ be a positive integer and rearrange  $L_{1}, \ldots, L_{d_1}, R_{1}, \cdots, R_{d_2}$ as  $L_{1}^{\prime}, \ldots, L_{d_1+d_2-d}^{\prime}, R_{1}^{\prime}, \ldots, R_{d}^{\prime}$ ($R_{1}^{\prime}, \ldots, R_{d_1+d_2}^{\prime}$ for $d_1+d_2\le d$) such that $\tau(S(L_{1}^{\prime}),x)\ge \ldots\ge \tau(S(L_{d_1+d_2-d}^{\prime}),x)\ge\tau (S(R_{1}^{\prime}),x)\ge\ldots\ge \tau( S(R_{d}^{\prime}),x)$.  Let $T_1^{\prime}$ be the  tree with root $v_1$ and   the branches $L_{1}^{\prime}, \ldots, L_{d_1+d_2-d}^{\prime}$  for $d_1+d_2>d$ ($T_1^{\prime}$ is an isolate vertex for $d_1+d_2\le d$ ) and  $T_2^{\prime}$ be the  tree with root $v_2$ and   the branches $R_{1}^{\prime}, \cdots, R_{d}^{\prime}$ for $d_1+d_2>d$ (the branches $R_{1}^{\prime}, \ldots, R_{ d_1+d_2}^{\prime}$ for $d_1+d_2\le d)$. Let $T^{\prime}$ be the tree obtained  from $T_1^{\prime}, T_2^{\prime}, T_0$ by identifying $v_1$ and $u$, $v_2$ and $v$, respectively (see Fig.1).
If $M_{10}(S(T_0),x)\le M_{01}(S(T_0),x)$, where $M_{10}(S(T_0),x)$ and $ M_{01}(S(T_0),x)$  are matching generating functions of  saturate $u$ but not $v$,  saturate $v$ but not $u$, respectively, then
\begin{equation}\label{exchange10}
M(S(T),x)\ge M(S(T^{\prime}),x).
\end{equation}
 Further (\ref{exchange10}) becomes equality if and only if
$$\max\{\tau(S(R_{j}),x): 1\le j\le d_2\}\le  \min\{\tau(S(L_{j}),x): 1\le j\le d_1 \}$$
with $d_2=d$,
or $M_{10}(S(T_0),x)= M_{01}(S(T_0),x)$ and
 $$\max\{\tau(S(L_{j}),x): 1\le j\le d_1\}\le  \min\{\tau(S(R_{j}),x): 1\le j\le d_2\}$$
 with $d_1=d$.
\end{theorem}
 $$
\begin{tikzpicture}[scale=0.6]\label{fig1}
\coordinate[label=$u$] (I) at (3cm,2cm);
\coordinate[label=$v$] (I) at (9cm,2cm);
\coordinate[label=$\vdots$] (I) at (-1cm,2cm);
\coordinate[label=$\vdots$] (I) at (-1cm,1cm);
\coordinate[label=$\vdots$] (I) at (13cm,2cm);
\coordinate[label=$\vdots$] (I) at (13cm,1cm);
\coordinate[label=$\cdots$] (I) at (6cm,2.5cm);
\coordinate[label=$T_0$] (I) at (6cm,-0cm);
\coordinate[label=$T_1$] (I) at (2.5cm,-2.5cm);
\coordinate[label=$T_2$] (I) at (9.4cm,-2.5cm);
\coordinate[label=$L_1$] (I) at (-2cm,6.5cm);
\coordinate[label=$L_2$] (I) at (-2cm,4cm);
\coordinate[label=$L_{d_1}$] (I) at (-2cm,-1.5cm);
\coordinate[label=$R_1$] (I) at (14cm,6.5cm);
\coordinate[label=$R_2$] (I) at (14cm,4cm);
\coordinate[label=$R_{d_2}$] (I) at (14cm,-1.6cm);
\coordinate[label=$u_1$] (I) at (-1cm,7cm);
\coordinate[label=$u_2$] (I) at (-1cm,4.5cm);
\coordinate[label=$u_{d_1}$] (I) at (-0.9cm,-2.1cm);
\coordinate[label=$v_1$] (I) at (13cm,7cm);
\coordinate[label=$v_2$] (I) at (13cm,4.5cm);
\coordinate[label=$v_{d_2}$] (I) at (13cm,-2cm);
\fill[black]
(3cm,3cm) circle(1mm)
(9cm,3cm) circle(1mm);
\draw[-]
(3cm,3cm)--(-1cm,7cm)
(-1cm,7cm)--(-2.5cm,8cm)--(-2.5cm,6cm)--cycle
(9cm,3cm)--(13cm,7cm)
(13cm,7cm)--(14.5cm,8cm)--(14.5cm,6cm)--cycle
(3cm,3cm)--(-1cm,4.5cm)
(-1cm,4.5cm)--(-2.5cm,5.5cm)--(-2.5cm,3.5cm)--cycle
(9cm,3cm)--(13cm,4.5cm)
(13cm,4.5cm)--(14.5cm,5.5cm)--(14.5cm,3.5cm)--cycle
(3cm,3cm)--(-1cm,-1cm)
(-1cm,-1cm)--(-2.5cm,-2cm)--(-2.5cm,0cm)--cycle
(9cm,3cm)--(13cm,-1cm)
(13cm,-1cm)--(14.5cm,-2cm)--(14.5cm,0cm)--cycle;
\draw[dashed]
(3.2cm,-2.5cm)--(3.2cm,8.5cm)--(-3cm,8.5cm)--(-3cm,-2.5cm)--cycle
(8.8cm,-2.5cm)--(8.8cm,8.5cm)--(15cm,8.5cm)--(15cm,-2.5cm)--cycle;
\draw (6,3) ellipse (4cm and 2cm);
\end{tikzpicture}$$
$$Fig.~1, \ {\rm  Tree} \  T$$

\begin{proof}
Let
\begin{itemize} \setlength{\itemsep}{-\itemsep}
\item $m_{00}(S(T_0),k)$ be the number of matchings of $T_0$ of cardinality $k$ which saturate neither $u$ nor $v$;
\item $m_{10}(S(T_0),k)$ be the number of matchings of $T_0$ of cardinality $k$ which saturate $u$, but not $v$;
\item $m_{01}(S(T_0),k)$ be the number of matchings of $T_0$ of cardinality $k$ which saturate $v$, but not $u$;
\item $m_{11}(S(T_0),k)$ be the number of matchings of $T_0$ of cardinality $k$ which saturate both $u$ and $v$.
    \end{itemize}
By Lemma~\ref{subdivision-rec}, we have
\begin{footnotesize}
\begin{eqnarray*}
M(S(T),x)&=&M_{00}(S(T_0),x)M(S(T_1),x)M(S(T_2),x)+M_{10}(S(T_0),x)M_0(S(T_1),x)M(S(T_2),x)\\
&+&M_{01}(S(T_0))M(S(T_1),x)M_0(S(T_2))+M_{11}(S(T_0))M_0(S(T_1),x)M_0(S(T_2),x)\\
&=& \prod_{j=1}^{d_1} (M(S(L_{j}),x)+xM_0(S(L_{j}),x))\prod_{j=1}^{d_2}(M(S(R_{j}),x)+xM_0(S(R_{j}),x))\\
&\times &\left\{M_{00}(S(T_0),x)\left(1+\sum_{j=1}^{d_1}\frac{x}{1+x\tau(S(L_{j}),x)}\right)
\left(1+\sum_{j=1}^{d_2}\frac{x}{1+x\tau(S(R_{j}),x)}\right)+M_{11}(S(T_0))\right.
\\
&+&\left.  M_{10}(S(T_0),x)(1+\sum_{j=1}^{d_2}\frac{x}{1+x\tau(S(R_{j}),x)}) + M_{01}(S(T_0),x)(1+\sum_{j=1}^{d_1}\frac{x}{1+x\tau(S(L_{j}),x)}) \right\}.\\
\end{eqnarray*}
\end{footnotesize}
On the other hand,
\begin{footnotesize}
\begin{eqnarray*}
M(S(T^{\prime}),x)&=&M_{00}(S(T_0),x)M(S(T_1^{\prime}),x)M(S(T_2^{\prime}),x)+M_{10}(S(T_0),x)M_0(S(T_1^{\prime}),x)M(S(T_2^{\prime}),x)\\
&+&M_{01}(S(T_0))M(S(T_1^{\prime}),x)M_0(S(T_2^{\prime}),x)+M_{11}(S(T_0),x)M_0(S(T_1^{\prime}),x)M(S(T_2^{\prime}),x)\\
&=& \prod_{j=1}^{d_1+d_2-d} (M(S(L_{j}^{\prime}),x)+xM_0(S(L_{j}^{\prime}),x))\prod_{j=1}^{d}(M(S(R_{j}^{\prime}),x)
+xM_0(S(R_{j}^{\prime}),x))\\
&\times &\left\{M_{00}(S(T_0),x)\left(1+\sum_{j=1}^{d_1+d_2-d}\frac{x}{1+x
\tau(S(L_j^{\prime}),x)}\right)\left( 1+\sum_{j=1}^{d}\frac{x}{1+x
\tau(S(R_j^{\prime}),x)} \right)+M_{11}(S(T_0),x)
 \right.\\
&+& \left. M_{10}(S(T_0),x)\left(1+\sum_{j=1}^d\frac{1}{1+\tau(S(R_{j}^{\prime}),x)}\right)+ M_{01}(S(T_0),x)\left(1+\sum_{j=1}^{d_1+d_2-d}\frac{x}{1+x\tau(S(L_{j}^{\prime}),x)}\right)\right\}.\\
\end{eqnarray*}
\end{footnotesize}
 Since $L_{1}^{\prime}, \ldots, L_{d_1+d_2-d}^{\prime}, R_{1}^{\prime}, \ldots, R_{d}^{\prime}$ is a rearrangement of $L_{1}, \ldots, L_{d_1}, R_{1}, \ldots, R_{d_2}$ and $\tau(S(L_{1}^{\prime}),x)\ge \ldots\ge \tau(S(L_{d_1+d_2-d}^{\prime}),x)\ge\tau (S(R_{1}^{\prime}),x)\ge\ldots\ge \tau( S(R_{d}^{\prime}),x),$  we have  \begin{eqnarray}\label{eq8-1}
&&\prod_{j=1}^{d_1+d_2-d}(M(S(L_{j}^{\prime}),x)+xM_0(S(L_{j}^{\prime}),x))\prod_{j=1}^{d}(M(S(R_{j}^{\prime}),x)
+xM_0(S(R_{j}^{\prime}),x))\nonumber\\
&=&\prod_{j=1}^{d_1} (M(S(L_{j}),x)+xM_0(S(L_{j}),x))\prod_{j=1}^{d_2}(M(S(R_{j}),x)+xM_0(S(R_{j}),x))
\end{eqnarray}
and

\begin{equation}\label{eqn20}
\sum_{j=1}^{d_1+d_2-d}\frac{x}{1+x\tau(S(L_{j}^{\prime}),x)}+
\sum_{j=1}^d\frac{x}{1+x\tau(S(R_{j}^{\prime}),x)}=
\sum_{j=1}^{d_1}\frac{x}{1+x\tau(S(L_{j}),x)}+\sum_{j=1}^{d_2}\frac{x}{1+x\tau(S(R_{j}),x)}.
\end{equation}
Further
$$\frac{x}{1+x\tau(S(L_{1}^{\prime}),x)}\le \ldots\le \frac{x}{1+x\tau(S(L_{d_1+d_2-d}^{\prime}),x)}
\le \frac{x}{1+x\tau(S(R_{1}^{\prime}),x)}\le\ldots\le \frac{x}{1+x\tau(S(R_{d}^{\prime}),x)},$$
yields
 \begin{equation}\label{eq7}
\sum_{j=1}^{d_1}\frac{x}{1+x\tau(S(L_{j}),x)}\ge
\sum_{j=1}^{d_1+d_2-d}\frac{x}{1+x\tau(S(L_{j}^{\prime}),x)}
\end{equation} and
\begin{equation}\label{eq8}
\sum_{j=1}^{d_2}\frac{x}{1+x\tau(S(R_{j}),x)}\le
\sum_{j=1}^d\frac{x}{1+x\tau(S(R_{j}^{\prime}),x)}.
\end{equation}
Moreover, (\ref{eq8}) becomes equality if and only if $d_2=d$ and
$$\max\{\tau(S(R_{j}),x): 1\le j\le d_2\}\le  \min\{\tau(S(L_{j}),x): 1\le j\le d_1\}.$$
Then (\ref{eqn20}), (\ref{eq7}) and (\ref{eq8}) yield
\begin{footnotesize}
\begin{eqnarray}\label{eq9}
&&\left(1+\sum_{j=1}^{d_1}\frac{x}{1+x\tau(S(L_{j}),x)}\right)
\left(1+\sum_{j=1}^{d_2}\frac{x}{1+x\tau(S(R_{j}),x)}\right)\nonumber\\
&-&\left(1+\sum_{j=1}^{d_1+d_2-d}\frac{x}{1+x\tau(S(L_{j}^{\prime}),x)}\right)
\left(1+\sum_{j=1}^{d}\frac{x}{1+x\tau(S(R_{j}^{\prime}),x)}\right)\nonumber\\
&=&
\left(\sum_{j=1}^{d}\frac{x}{1+x\tau(S(R_{j}^{\prime}),x)}-
\sum_{j=1}^{d_1}\frac{x}{1+x\tau(S(L_{j}),x)}\right)
\left(\sum_{j=1}^{d}\frac{x}{1+x\tau(S(R_{j}^{\prime}),x)}
-\sum_{j=1}^{d_2}\frac{x}{1+x\tau(S(R_{j}),x)}\right)\nonumber\\
&\ge& 0
\end{eqnarray}
\end{footnotesize}
Moreover,
\begin{footnotesize}
\begin{eqnarray}\label{eq10}
&&M_{10}(S(T_0),x)\left(1+\sum_{j=1}^{d_2}\frac{x}{1+x\tau(S(R_{j}),x)}\right) + M_{01}(S(T_0),x)\left(1+\sum_{j=1}^{d_1}\frac{x}{1+x\tau(S(L_{j}),x)}\right)\nonumber\\
&-&\left\{
M_{10}(S(T_0),x)\left(1+\sum_{j=1}^d\frac{1}{1+\tau(S(R_{j}^{\prime}),x)}\right)+ M_{01}(S(T_0),x)\left(1+\sum_{j=1}^{d_1+d_2-d}\frac{x}{1+x\tau(S(L_{j}^{\prime}),x)}\right)
\right\}\nonumber\\
&=&\left(M_{01}(S(T_0),x)-M_{10}(S(T_0),x)\right)\left(\sum_{j=1}^d\frac{x}{1+x\tau(S(R_{j}^{\prime}),x)}
-
\sum_{j=1}^{d_2}\frac{x}{1+x\tau(S(R_{j}),x)}\right)\ge 0.
\end{eqnarray}
\end{footnotesize}
Hence by (\ref{eq8-1}),(\ref{eq9}) and (\ref{eq10}), we have
 $$M(S(T),x)-M(S(T^{\prime}),x)\ge 0.$$
 Further, if  $M(S(T),x)-M(S(T^{\prime}),x)= 0,$ then
  (\ref{eq9}) and (\ref{eq10}) become equalities.
  Hence $$\sum_{j=1}^d\frac{x}{1+x\tau(S(R_{j}),x)}=
\sum_{j=1}^{d_2}\frac{x}{1+x\tau(S(R_{j}^{\prime}),x)},$$ or
$$M_{10}(S(T_0),x)=M_{01}(S(T_0),x) ~\mbox{and}~ \sum_{j=1}^d\frac{x}{1+x\tau(S(R_{j}),x)}=
\sum_{j=1}^{d_1}\frac{x}{1+x\tau(S(L_{j}^{\prime}),x)}.$$
Therefore, $$\max\{\tau(S(R_{j}),x): 1\le j\le d_2\}\le  \min\{\tau(S(L_{j}),x): 1\le j\le d_1 \} \mbox{
with $d_2=d$,}$$
or $M_{10}(S(T_0),x)= M_{01}(S(T_0),x)$ and
 $$\max\{\tau(S(L_{j}),x): 1\le j\le d_1\}\le  \min\{\tau(S(R_{j}),x): 1\le j\le d_2\}  ~~\mbox{with $d_1=d$.}$$
This completes the proof.
 \end{proof}

Let $\mathcal{T}_{n, d+1}$ be the set of all trees of order $n$ with the maximum degree $d+1$ and  $\mathcal{S(T)}_{n, d+1}$ be the set of the subdivision trees of any tree $T$ in $\mathcal{T}_{n, d+1}$.  A tree $S(\widetilde{T})$  in $\mathcal{S(T)}_{n, d+1}$ is called an {\it optimal tree} if $M(S(T),x)\ge M(S(\widetilde{T}),x)$ for any $S(T)\in \mathcal{S(T)}_{n, d+1}$ and $x>0$.
\begin{corollary}\label{exchange} Let $S(\widetilde{T})$ be  an optimal tree in $\mathcal{S(T)}_{n, d+1}$. If $\widetilde{T}$ can be decomposed as $T_1,T_2$ and $T_0$ as $Fig.~1$. If $u$ and $v$ are non-pendent vertices
and $\tau(S(L_1),x)>\tau(S(R_1),x), $ then   $d_2= d$ and
\begin{equation}\label{max-min}
\min\{\tau(S(L_i),x): 1\le i\le d_1\}\ge \max\{\tau(S(R_i),x): 1\le i\le d_2\}.
\end{equation}
\end{corollary}
\begin{proof} We follows the symbols in Theorem~\ref{exchange0}. First we have the following claim
\begin{equation}\label{coreq11}
M_{10}(S(T_0),x)\le M_{01}(S(T_0),x)
\end{equation}
In fact, suppose $M_{10}(S(T_0),x)>M_{01}(S(T_0),x)$.
Rearrange  $L_{1}, \ldots, L_{d_1}, R_{1}, \ldots, R_{d_2}$ as  $L_{1}^{\prime}, \ldots, L_{d_1+d_2-d}^{\prime}, R_{1}^{\prime}, \ldots, R_{d}^{\prime}$ ($R_{1}^{\prime}, \ldots, R_{d_1+d_2}^{\prime}$ for $d_1+d_2\le d$) such that $\tau(S(L_{1}^{\prime}),x)\ge \ldots\ge \tau(S(L_{d_1+d_2-d}^{\prime}),x)\ge\tau (S(R_{1}^{\prime}),x)\ge\ldots\ge \tau( S(R_{d}^{\prime}),x)$.
 Let $T_1^{\prime\prime}$ be the  tree with root $v_1$ and  the branches $R_{1}^{\prime}, \ldots, R_{d}^{\prime}$  for $d_1+d_2>d$ (the branches $R_{1}^{\prime}, \ldots, R_{d_1+d_2}^{\prime}$ for $d_1+d_2\le d$) and  $T_2^{\prime\prime}$ be the  tree with root $v_2$ and   the branches $L_{1}^{\prime}, \cdots, L_{d_1+d_2-d}^{\prime}$ for $d_1+d_2>d$ (no branches for $d_1+d_2\le d)$. Then Let $T^{\prime\prime}$ be the tree obtained  from $T_1^{\prime\prime}, T_2^{\prime\prime}, T_0$ by identifying $v_1$ and $u$, $v_2$ and $v$, respectively.
 By Theorem~\ref{exchange0}, $M(S(\widetilde{T}), x)\ge M(S(T^{\prime\prime}), x)$. On the other hand, since
 $S(T^{\prime\prime})\in \mathcal{S(T)}_{n, d+1}$ and $S(\widetilde{T})$ is an optimal tree in  $\mathcal{S(T)}_{n, d+1}$, we have $M(S(T^{\prime\prime}), x)\ge M(S(\widetilde{T}), x)$. Hence $M(S(\widetilde{T}), x)= M(S(T^{\prime\prime}), x)$. By Theorem~\ref{exchange0} again,   we have $$\max\{\tau(S(L_{j}),x): 1\le j\le d_1\}\le  \min\{\tau(S(R_{j}),x): 1\le j\le d_2 \},$$
 which contradicts to the condition $\tau(S(L_1),x)>\tau(S(R_1),x), $. Hence the claim holds.
Hence by Theorem~\ref{exchange0}, the corollary holds.
\end{proof}
 The following Corollary~\ref{cor27} is easily obtained from Theorem~\ref{exchange0}.
\begin{corollary}\label{cor27}
 Let $S(\widetilde{T})$ be  an optimal tree in $\mathcal{S(T)}_{n, d+1}$. Then there exists  at most one vertex $u$ with degree $2\le deg(u)\le d$.
 \end{corollary}

\section{Proof of Theorem~\ref{main1}}
In order to prove Theorem~\ref{main1}, we need some  Lemmas.

\begin{lemma}\label{complete tree}
$$
\tau(S(C_1),x)=1,\ \ \tau(S(C_2),x)=\frac{1+x}{(d+1)x+1},$$
\begin{equation}\label{ch-ch-1}
\tau(S(C_h),x)=\frac{1}{1+\frac{dx}{1+x\tau(S(C_{h-1}),x)}}.\end{equation}
Further $\tau(S(C_h),x)<\tau(S(C_{h-1}),x)$ for $h\ge 2$.
\end{lemma}
\begin{proof} Clearly, 
$\tau(S(C_1),x)=\tau(C_1,x)=1.$
For $h\ge 2$, by (\ref{LM-tau}), it is easy to see that
$$\tau(S(C_h),x)=\frac{1}{1+\frac{dx}{1+x\tau(S(C_{h-1}),x)}}.$$
Further,
$$\tau(S(C_2),x)-\tau(S(C_1),x)=\frac{1+x}{1+(d+1)x}-1=-\frac{dx}{1+(d+1)x}<0.$$
Then, for $h>2$,
\begin{eqnarray*}
&&\tau(S(C_h),x)-\tau(S(C_{h-1}),x)=\frac{1}{1+\frac{dx}{1+x\tau(S(C_{h-1}),x)}}-
\frac{1}{1+\frac{dx}{1+x\tau(S(C_{h-2}),x)}}\\
&=&\frac{dx^2\left(\tau(S(C_{h-1}),x)-\tau(S(C_{h-2}),x)\right)}{(1+x\tau(S(C_{h-1}),x)(1+x\tau(S(C_{h-2}),x)
\left(1+\frac{dx}{1+x\tau(S(C_{h-1}),x)}\right)\left(1+\frac{dx}{1+x\tau(S(C_{h-2}),x) }\right)}<0.\\
\end{eqnarray*}
This completes the proof.
\end{proof}

Let $u$ be any vertex in a tree $T=(V(T), E(T))$. Denote by $N(u)$ the set of all vertices adjacent to  $u$, i.e., $N(u):=\{v\in V(T) \ |\ uv\in E(T)\}$.

\begin{lemma}\label{exchange-cor-d}
Let $S(\widetilde{T})$ be an optimal tree in $\mathcal{S(T)}_{n, d+1}$.  If  there exists a vertex $u$ with degree $2\le deg(u)\le d$ in $\widetilde{T}$, then there are $deg(u)-1$ pendent vertices in   $N(u)$.  Further, if there exists a vertex $v\neq u$ such that
there are $1\le k\le d-1$ pendent vertices in  $N(v)$, then $uv\in E(\widetilde{T})$, and  there are $d$ pendent vertices or no pendent vertices in  $N(w)$ for any $w\neq u, v.$
\end{lemma}

\begin{proof} By Corollary~\ref{cor27} that $u$ is a unique vertex. Let $y$ be the farthest non-pendent vertex from $u$ in $\widetilde{T}$. Then there is exact one non-pendent vertex in $N(y)$  and  $\widetilde{T}$ can be decomposed as
$T_0$, $T_1$ and $T_2$ (see Fig.1), where $T_1$ has root $y$ with the branches $C_1, \ldots, C_1$ and $T_2$ has  root $u$ with the branches $L_1, \cdots, L_{deg(u)-1}$.
By Lemma~\ref{subgraph}, $\tau(S(C_1), x)\ge \tau(S(L_j), x) $ with equality if and only if $L_j=C_1$ for $1\le j\le deg(u)-1$.  By Corollary~\ref{exchange} and $deg(u)-1<d$,
we have $\tau(S(C_1), x)=\tau(S(L_j), x) $. Hence $L_j=C_1$ for $1\le j\le deg(u)-1$. Hence there are  $deg(u)-1$ pendent vertices in $N(u)$.

 Suppose that $uv\notin E(\widetilde{T})$. Then there exists a vertex $z$  with $uz\in E(\widetilde{T})$ such that $\widetilde{T}$ can be decomposed as $T_0$, $T_1$ and $T_2$, where $
T_0$ contains vertices $z, v$, not $u$, $T_1$ has root $z$ and the branches $L_1$ containing $u$,  $L_2, \ldots, L_d$, and $T_2$ has root $v$ and the branches $R_1$ containing $C_2$, $R_2=C_1$, $R_3, \cdots, R_d$. By Lemma~\ref{subgraph},  $\tau(S(L_1), x)>\tau(S(C_2), x)\ge \tau(S(R_1), x)$. Hence by Corollary~\ref{exchange},
$$1>\tau(S(L_1),x)\ge \min \{\tau(S(L_i),x),~1\le i\le d\}\ge \max\{\tau(S(R_i),x),~1\le i\le d\}=1,$$
which is a contradiction. So $uv\in E(\widetilde{T})$.
Suppose that there exists another  vertex $w\neq u,v $ such that  there are $1\le p\le d$ pendent vertices  in $N(w)$.  Then  $\widetilde{T}$ can be decomposed as  trees $T_0$, $T_1$ and $T_2$, where $
T_0$ contains vertices $ v, w$,  $T_1$ has root $v$ and the branches $L_1=C_1$, $L_2\neq C_1$, $L_3, \cdots, L_d$, and $T_2$ has root $w$ and the branches $R_1=C_1$, $R_2\neq C_1$, $R_3,  \cdots, R_d$. Clearly $\tau(S(L_1), x)>\tau(S(R_2), x)$. Hence
$$1>\min \{\tau(S(L_i),x),~1\le i\le d\}\ge \max\{\tau(S(R_i),x),~1\le i\le d\}=1,$$
which is a contradiction.
 Hence  for any $w\neq u, v,$ there are $d$ pendent vertices or no pendent vertices in  $N(w)$.
\end{proof}

\begin{lemma}\label{degeled}
Let $S(\widetilde{T})$ be an optimal tree  in $\mathcal{S(T)}_{n, d+1}$. If  there is a vertex $v$ of $V(\widetilde{T})$ with degree $2\le deg(v)\le d$, then $\widetilde{T}$ is greedy tree $T_{d+1}^*$.
\end{lemma}
\begin{proof} By Corollary~\ref{cor27}, the degree of any vertex $u\neq v$ is $d+1$ or 1.
Let $P$ be a longest path with end vertex $v$ in $\widetilde{T}$.
If the length of $P$ is odd, denote $P=v_1v_2\cdots v_k v_{k+1}w_k\cdots w_1w_0$ with $v_1=v$.
Moreover, let $L_{1}^t, \ldots, L_{d}^t, L_{d+1}^t$  be the $d+1$ branches of $\widetilde{T}-v_t$  such that $L_{1}^t $ contains $v_1$ but no $w_1$; $L_{2}^t, \ldots, L_{d}^t$ do not contain $v_1$ and $w_1$, $L_{d+1}^t$ contains $w_1$ but no $v_1$ for  $t=2, \ldots, k+1$.
Similarly, let  $R_{1}^t, \ldots,  R_{d}^t, R_{d+1}^t $ be the $d+1$ branches of $\widetilde{T}-w_t$ such that $R_{1}^t $ contains $w_1$  but no $v_1$; $R_{2}^t,\ldots, R_{d}^t$ do not contain $v_1$ and $w_1$, $R_{d+1}^t$ contains $v_1$ but no $w_1$  for  $t=2, \cdots, k$ (See Fig.~2). Then we have the following Claim.

$$
\begin{tikzpicture}[scale=0.7]\label{fig2}
\coordinate[label=$v_1$] (I) at (4cm,4.3cm);
\coordinate[label=$v_2$] (I) at (7cm,4.3cm);
\coordinate[label=$v_k$] (I) at (10cm,4.3cm);
\coordinate[label=$v_{k+1}$] (I) at (13cm,4.3cm);
\coordinate[label=$w_k$] (I) at (16cm,4.3cm);
\coordinate[label=$w_2$] (I) at (19cm,4.3cm);
\coordinate[label=$w_1$] (I) at (22cm,4.3cm);
\coordinate[label=$\vdots$] (I) at (3cm,3.5cm);
\coordinate[label=$\vdots$] (I) at (23cm,3.5cm);
\coordinate[label=$L_{1}^2$] (I) at (3.2cm,1.5cm);
\coordinate[label=$R_{1}^2$] (I) at (22.8cm,1.5cm);
\coordinate[label=$L_{1}^k$] (I) at (3.2cm,-1cm);
\coordinate[label=$R_{1}^k$] (I) at (22.8cm,-1cm);
\coordinate[label=$L_{1}^{k+1}$] (I) at (3.4cm,-2.5cm);
\coordinate[label=$L_{d+1}^{k+1}$] (I) at (22.3cm,-2.5cm);
\coordinate[label=$L_{2}^2$] (I) at (6cm,0.3cm);
\coordinate[label=$R_{d}^2$] (I) at (18.3cm,0.3cm);
\coordinate[label=$L_{d}^2$] (I) at (7.9cm,0.3cm);
\coordinate[label=$R_{2}^2$] (I) at (20cm,0.3cm);
\coordinate[label=$L_{2}^k$] (I) at (9.1cm,0.3cm);
\coordinate[label=$R_{d}^k$] (I) at (15.1cm,0.3cm);
\coordinate[label=$L_{d}^k$] (I) at (10.9cm,0.3cm);
\coordinate[label=$R_{2}^k$] (I) at (16.9cm,0.3cm);
\coordinate[label=$L_{2}^{k+1}$] (I) at (12cm,-2.5cm);
\coordinate[label=$L_{d}^{k+1}$] (I) at (14cm,-2.5cm);
\coordinate[label=$\cdots$] (I) at (7cm,2cm);
\coordinate[label=$\cdots$] (I) at (10cm,2cm);
\coordinate[label=$\cdots$] (I) at (13cm,2cm);
\coordinate[label=$\cdots$] (I) at (16cm,2cm);
\coordinate[label=$\cdots$] (I) at (19cm,2cm);
\fill[black]
(3cm,5cm) circle(1mm)
(3cm,3cm) circle(1mm)
(4cm,4cm) circle(1mm)
(7cm,4cm) circle(1mm)
(10cm,4cm) circle(1mm)
(13cm,4cm) circle(1mm)
(16cm,4cm) circle(1mm)
(19cm,4cm) circle(1mm)
(22cm,4cm) circle(1mm)
(23cm,5cm) circle(1mm)
(23cm,3cm) circle(1mm);
\draw[-]
(4cm,4cm)--(3cm,5cm)
(4cm,4cm)--(3cm,3cm)
(22cm,4cm)--(23cm,5cm)
(22cm,4cm)--(23cm,3cm)
(4cm,4cm)--(7cm,4cm)
(10cm,4cm)--(13cm,4cm)--(16cm,4cm)
(19cm,4cm)--(22cm,4cm)
(7cm,4cm)--(6cm,3cm)--(5.5cm,1.5cm)--(6.4cm,1.5cm)--(6cm,3cm)
(7cm,4cm)--(8cm,3cm)--(7.5cm,1.5cm)--(8.4cm,1.5cm)--(8cm,3cm)
(10cm,4cm)--(9cm,3cm)--(8.6cm,1.5cm)--(9.5cm,1.5cm)--(9cm,3cm)
(10cm,4cm)--(11cm,3cm)--(10.5cm,1.5cm)--(11.4cm,1.5cm)--(11cm,3cm)
(13cm,4cm)--(12cm,3cm)--(11.6cm,-1.2cm)--(12.5cm,-1.2cm)--(12cm,3cm)
(13cm,4cm)--(14cm,3cm)--(13.5cm,-1.2cm)--(14.4cm,-1.2cm)--(14cm,3cm)
(16cm,4cm)--(15cm,3cm)--(14.6cm,1.5cm)--(15.5cm,1.5cm)--(15cm,3cm)
(16cm,4cm)--(17cm,3cm)--(16.5cm,1.5cm)--(17.4cm,1.5cm)--(17cm,3cm)
(19cm,4cm)--(18cm,3cm)--(17.6cm,1.5cm)--(18.5cm,1.5cm)--(18cm,3cm)
(19cm,4cm)--(20cm,3cm)--(19.5cm,1.5cm)--(20.4cm,1.5cm)--(20cm,3cm);
\draw[dashed]
(7cm,4cm)--(10cm,4cm)
(16cm,4cm)--(19cm,4cm)
(4.5cm,5.3cm)--(2.7cm,5.3cm)--(2.7cm,2.7cm)--(4.5cm,2.7cm)--cycle
(23.3cm,5.3cm)--(21.5cm,5.3cm)--(21.5cm,2.7cm)--(23.3cm,2.7cm)--cycle
(8.5cm,5.4cm)--(2.6cm,5.4cm)--(2.6cm,0cm)--(8.5cm,0cm)--cycle
(17.5cm,5.4cm)--(23.5cm,5.4cm)--(23.5cm,0cm)--(17.5cm,0cm)--cycle
(11.5cm,5.5cm)--(2.5cm,5.5cm)--(2.5cm,-1.2cm)--(11.5cm,-1.2cm)--cycle
(14.5cm,5.5cm)--(23.5cm,5.5cm)--(23.5cm,-1.2cm)--(14.5cm,-1.2cm)--cycle;
\end{tikzpicture}$$
$$Fig.~2~~~\mbox{ tree $\widetilde{T}$}$$
{\bf Claim :}
\begin{itemize} \setlength{\itemsep}{-\itemsep}
\item (i).
 $\tau(S(C_t),x)<\tau(S(L_{1}^t),x)< \tau(S(C_{t-1}),x),~2\le t\le k+1$.
\item (ii). $L_{i}^t=C_{t-1}$ or $L_{i}^t=C_{t},~2\le t\le k+1,~2\le i\le d$;
\item  (iii). $R_{1}^t=R_2^{t}=\cdots=R_{d}^t=C_t,~2\le t\le k$.
\end{itemize}

 We prove the claim by the induction on $t$.  For $t=2$, by Lemma~\ref{complete tree}, we have
\begin{eqnarray*}
\tau(S(C_1),x)=1> \tau(S(L_{1}^2),x)&=&\frac{1}{1+\frac{xdeg(v)}{1+x}}
>\frac{1}{1+\frac{dx}{1+x}}=\tau(S(C_2),x)=\tau(S(R_{1}^2),x).
\end{eqnarray*}
So (i) holds. By Lemma~\ref{exchange-cor-d}, for any vertex $w\in V(L_{d+1}^2)$, there are $d$ pendent vertices or no pendent vertices in $N(w)$. Thus $R_{1}^2=C_2$, which implies that $\tau(S(L_{1}^2),x)>\tau(S(R_{1}^2),x)$. Hence Corollary~\ref{exchange}, we have $\tau(S(L_{j}^2),x)\ge \tau(S(C_2),x)$ for $j=2, \cdots, d$. Therefore by Lemmas~\ref{subgraph} and \ref{exchange-cor-d}, we have $L_j^2=C_1$ or $C_2$  for $j=2, \cdots, d$.  So (ii) holds. Since $P$ is a longest path with ends $v_1$ and $w_0$, the distance between any vertex in $R_j^2$ and $w_2$ is no more than 2. By $\tau(S(R_j^2),x)\le \tau(S(L_1^2),x)$ and  Lemma~\ref{exchange-cor-d}, we have $R_j^2=C_2$ for $j=2, \cdots, d$. So (iii) holds.

Suppose that the claim holds for  less than $t$. By the induction prothesis,
$\tau(S(C_{t-1}),x)< \tau(S(L_{1}^{t-1}),x)< \tau(S(C_{t-2}),x)$
and $\tau(S(C_{t-1}),x)\le \tau(S(L_{j}^{t-1}),x)\le \tau(S(C_{t-2}),x)$  for $j=2, \cdots, d$. It follows from   $(\ref{LM-tau})$ and (\ref{ch-ch-1}) that
\begin{eqnarray*}
\tau(S(C_{t}),x)&=&
\frac{1}{1+\sum_{i=1}^d\frac{x}{1+x\tau(S(C_{t-1}),x)}}
< \frac{1}{1+\sum_{i=1}^d\frac{x}{1+x\tau(S(L_{i}^{t-1}),x)}}\\
&=&\tau(S(L_{1}^t),x)
<\frac{1}{1+\sum_{i=1}^d\frac{x}{1+x\tau(S(C_{t-2}),x)}}
=\tau(S(C_{t-1}),x)
\end{eqnarray*}
 Hence (i) holds for $t$.
 In order to prove (ii) holding for $t$, we first prove the following several claims.

 {\bf Claim~2.1:} $\tau(S(C_t), x)\le \tau(S(L_i^t), x)\le \tau(S(C_{t-1}, x)$ for $i=2, \cdots, d$.

  In fact, there are $d$ branches $L_1^t, \cdots, L_d^t $ containing no $w_t$ in $\widetilde{T}-v_t$ and there are $d$ branches $R_1^t=C_t, R_2^t,\cdots, R_d^t$ containing no $v_t$ in
     $\widetilde{T}-w_t$. By  Claim (i),  we have $\tau(S(C_t),x)<\tau(S(L_{1}^t),x)$. On the other hand, by the induction hypothesis, $R_1^{t-1}=\ldots=R_d^{t-1}=C_{t-1}$ which implies $R_1^t=C_t$.  Hence by Corollary~\ref{exchange}, $\min\{\tau(S(L_1^t), x), \cdots, \tau(S(L_d^t), x)\}\ge
\tau(S(C_t),x)$.  Further  there are $d$ branches $L_1^t, \cdots, L_d^t $ containing no $w_{t-1}$ in $\widetilde{T}-v_t$ and there are $d$ branches $R_1^{t-1}=C_{t-1}, \ldots, R_d^{t-1}=C_{t-1}$ containing no $v_t$ in      $\widetilde{T}-w_{t-1}$.  By Claim (i), $\tau(S(L_{1}^t),x)<\tau(S(C_{t-1}),x)=\tau(S(R_1^{t-1}),x).$ Hence by Corollary~\ref{exchange}, $\max\{S(L_{2}^t),x), \cdots, S(L_{d}^t),x)\}\le \tau(S(C_{t-1}),x).$ So Claim~2.1 holds.

 Let  $l+1$ be the maximum distance  between $v_t$ and any vertex in $\bigcup_{i=1}^tV(L_i^t)$. Then $l\ge t-1$.  Denote by
 $$V_j=\{u\ |\  dist(u, v_t)=l-j+1, u\in \bigcup_{i=1}^d V(L_i^t)\},\  j=0, \ldots, l.$$

     {\bf Claim 2.2:} For any $u\in V_{l-j}$, there are $d$ branches
     $L_1^u, \cdots, L_d^u$ containing no $v_t$ in $\widetilde{T}-u$ such that
     \begin{equation}
     \tau(S(C_{t-j-1}), x)\le \tau(S(L_i^u), x) \le \tau(S(C_{t-j-2}), x)
     , \ \  i=1, \cdots, d,
     \end{equation} where $j=0, \cdots, \min\{t, l\}-2$.

     We prove Claim~2.2 by the induction on $j$.  Let  $L_1^u, \cdots, L_d^u$  be $d$ the branches
    containing no $v_{t}$ in $\widetilde{T}-u$ and  $T^u$ be the subtree of $T$ consisting of $u$ and $L_1^u, \cdots, L_d^u$. For $j=0$,
there clearly exists an $1\le p\le d$ such that $T^u=L_p^t$.   If there exists an $1\le i\le d$ such that  $ \tau(S(L_i^u), x) < \tau(S(C_{t-1}), x)$, let    $R_1^{t-1}=C_{t-1}, \cdots, R_d^{t-1}$  be $d$ the branches
    containing no $v_t$ in $\widetilde{T}-w_{t-1}$. Hence by Corollary~\ref{exchange},
    $\max\{  \tau(S(L_1^u), x), \cdots,  \tau(S(L_d^u), x)\}\le  \tau(S(C_{t-1}), x)$.
    Then
    $$ \tau(S(L_p^{t}), x)=\tau(S(T^u), x)=\frac{1}{1+\sum_{q=1}^d\frac{x}{1+x\tau(S(L_q^u), x)}}
    <\frac{1}{1+\sum_{q=1}^d\frac{x}{1+x\tau(S(C_{t-1}), x)}}= \tau(S(C_{t}), x),$$
     which contradicts  to Claim~2.1. Therefore,
     $$ \tau(S(C_{t-1}), x)\le \tau(S(L_i^u), x), i=1, \ldots, d.$$
     On the other hand, if there exists $1\le i\le d$ such that
     $ \tau(S(L_i^u), x) > \tau(S(C_{t-2}), x)$, let    $R_1^{t-2}=C_{t-2}, \cdots, R_d^{t-2}$  be $d$ the branches
    containing no $v_t$ in $\widetilde{T}-w_{t-2}$. By
    Corollary~\ref{exchange},
   $$\min\{  \tau(S(L_1^u), x), \cdots,  \tau(S(L_d^u), x)\}\le  \tau(S(C_{t-2}), x).$$
    Then
      $$ \tau(S(L_p^{t}), x)=\tau(S(T^u), x)=\frac{1}{1+\sum_{q=1}^d\frac{x}{1+x\tau(S(L_q^u), x)}}
    >\frac{1}{1+\sum_{q=1}^d\frac{x}{1+x\tau(S(C_{t-1}), x)}}= \tau(S(C_{t-1}), x),$$
     which contradicts to Claim~2.1. Hence Claim~2.2 holds for $j=0$.
     Now assume that Claim~2.2 holds for $j$ and consider the claim for $j+1$.
     For any $u\in V_{l-(j+1)}$,  let
      $L_1^u, \ldots, L_d^u$  be $d$ the branches containing no $v_t$ in $\widetilde{T}-u$  and  $T^u$ be the subtree  consisting of $u$ and $L_1^u, \ldots, L_d^u$.
Clearly there exists a  $u^{\prime}\in V_{l-j}$ such that there exists a branch $L_1^{u^{\prime}}$ in $T_{d+1}^*-u^{\prime}$   such that $T^u=L_1^{u^{\prime}}$.
     If there exists an $1\le i\le d$ such that  $ \tau(S(L_i^u), x) < \tau(S(C_{t-j-2}), x)$,    let $R_1^{t-j-2}=C_{t-j-2}, \cdots, R_d^{t-j-2}$  be $d$ the branches
    containing no $u$ in $\widetilde{T}-w_{t-j-1}$. By Corollary~\ref{exchange},
    $$\max\{  \tau(S(L_1^u), x), \cdots,  \tau(S(L_d^u), x)\}\le  \tau(S(C_{t-j-2}), x).$$
          Then
    $$ \tau(S(L_1^{u^{\prime}}), x)=\tau(S(T^{u}), x)=\frac{1}{1+\sum_{q=1}^d\frac{x}{1+x\tau(S(L_q^u), x)}}
    <\frac{1}{1+\sum_{q=1}^d\frac{x}{1+x\tau(S(C_{t-j-2}), x)}}= \tau(S(C_{t-j-1}), x),$$
     which contradicts to the induction hypothesis. Therefore, for any $u\in V_{l-j-1}$,
     $$ \tau(S(C_{t-j-2}), x)\le \tau(S(L_i^u), x), i=1, \cdots, d.$$
     On the other hand, if there exists $1\le i\le d$ such that
     $ \tau(S(L_i^u), x) > \tau(S(C_{t-j-3}), x)$.   Let    $R_1^{t-j-3}=C_{t-j-3}, \cdots, R_d^{t-j-3}$  be $d$ the branches
    containing no $u$ in $\widetilde{T}-w_{t-j-2}$. By Corollary~\ref{exchange},
    $$\min\{  \tau(S(L_1^u), x), \cdots,  \tau(S(L_d^u), x)\}\ge  \tau(S(C_{t-j-3}), x).$$
    Then
      $$ \tau(S(L_1^{u^{\prime}}), x)=\tau(S(T^u), x)=\frac{1}{1+\sum_{q=1}^d\frac{x}{1+x\tau(S(L_q^u), x)}}
    >\frac{1}{1+\sum_{q=1}^d\frac{x}{1+x\tau(S(C_{t-j-3}), x)}}= \tau(S(C_{t-j-2}), x),$$
     which contradicts to the induction hypothesis. Hence Claim~2.2 holds for $j+1$.
          Therefore Claim~2.2 holds.

   {\bf Claim~2.3:} $l=t-1$.
   If $l>t-1$, by Claim 2.2, for any $u\in V_{l-t+2},$
   $$  \tau(S(C_{1}), x) \le \tau(S(L_i^u),x), \ i=1, \cdots, d.$$
   On the other hand, there exists a $u^{\prime}\in V_{l-t+2}$ such that
   the largest distance between $u^{\prime}$ and the pendent vertex  is at least 2, then $C_2$ is a proper subgraph $L_1^{u^{\prime}}$, which implies   $\tau(S(L_1^{u^{\prime}}), x)\le  \tau(S(C_{2}), x)$. It is a contradiction.
   Hence $l\le t-1$.  In addition $l\ge t-1$, we have $l=t-1$.

    {\bf Claim 2.4:} For any $u\in V_{t-j-1}, ~j=0, \cdots, t-3$. Let $L_1^u, \cdots, \cdots, L_d^u$ be the $d$ branches containing no $v_t$ in $\widetilde{T}-u$ and $T^u$ consist of $u$ and $d$ branches  $L_1^u, \cdots, \cdots, L_d^u$. Then
    $L_1^u=\cdots=L_d^u=C_{t-j-1}$ or $L_1^u=\cdots=L_d^u=C_{t-j-2}$, i.e.,  $T^u=C_{t-j}$ or $T^u=C_{t-j-1}$.

     We prove Claim 2.4 by the induction for $t-j-1$. In fact, for $j=t-3$ and $u\in V_{2},$  by Claim 2.2, $\tau(S(C_{2}), x)\le \tau(S(L_i^u), x)\le \tau(S(C_{1}), x) $ for $ i=1, \cdots, d$. Hence
    $L_i^u=C_2$ or $L_i^u=C_1$ for $ i=1, \cdots, d$. If, say $L_1^u=C_2$ and $L_2^u=C_1$, then by  $\tau(S(L_1^2), x)>\tau(S(L_1^u),x)$ and Corollary~\ref{exchange},
    $\tau(S(L_1^2), x)\ge \max\{\tau(S(L_1^u), x), \cdots, \tau(S(L_d^u), x)\}\ge \tau(S(C_1), x)$, which is a contradiction. Hence $L_1^u=\cdots=L_d^u=C_{2}$ or
    $L_1^u=\cdots=L_d^u=C_{1}$, i.e.,  $T^u=C_{3}$ or $T^u=C_{2}$ for $u\in V_2$.
    Assume that Claim 2.4 hold for any vertex in $ V_{t-j-2}$. Now for $u\in   V_{t-j-1}$.
    Let $z_1, \cdots, z_d\in V_{t-j-1}$ be the roots of $L_1^u, \cdots, L_d^u$,
     respectively. By the induction hypothesis, $L_1^u, \cdots, L_d^u\in \{C_{t-j-1}, C_{t-j-2}\}$. Further $L_1^u=\cdots= L_d^u= C_{t-j-1}$ or $L_1^u=\cdots= L_d^u= C_{t-j-2}$. In fact, if, say  $L_1^u=C_{t-j-1}$ and $ L_2^u=C_{t-j-2}$, By
     $\tau(S(L_1^{t-j-1}),x)>\tau(S(C_{t-j-1}),x)$ and Corollary~\ref{exchange},
     $$\tau(S(L_1^{t-j-1}),x)\ge \max\{\tau(S(C_{t-j-1}),x), \tau(S(C_{t-j-2}),x)\ge \tau(S(C_{t-j-2}),x),$$
    which  contradicts to Claim 2.1.  Hence $L_1^u=\cdots= L_d^u= C_{t-j-1}$ or $L_1^u=\cdots= L_d^u= C_{t-j-2}$, i.e., $T^u=C_{t-j}$ or $T^u=C_{t-j-1}$.
So Claim 2.4 holds.

Hence $L_i^t=C_t$ or $L_i^t=C_{t-1}$ for $i=2, \cdots, d$. In other words, Claim (ii)  holds.

Similarly, we can prove Claim (iii) and  omit the detail.
It is easy from Claims that $\widetilde{T}$ is greedy tree.
If the length of $P$ is even,  we can prove this assertion by similar method. So we finish our proof.
\end{proof}

\begin{lemma}\label{d+1}
Let $S(\widetilde{T})$ be an optimal tree in $\mathcal{S(T)}_{n, d+1}$. If the degree of any vertex is 1 or $d+1$, then there exists at most one vertex $u$ such that there are $1\le k\le d-1$ pendent vertices in $N(u)$. Further suppose that there exists a vertex $u$ such that there are $1\le k\le d-1$ pendent vertices in $N(u)$. If the branches $T_1, \cdots, T_{d+1}$ of $\widetilde{T}-u$ contains no $C_3$, then $\widetilde{T}$ is greedy tree. If one of  the branches $T_1, \cdots, T_{d+1}$ of $T-u$ contains $C_3$, say $T_{d+1}\supseteq C_3$, then $T_i=C_1$ or $C_2$ for $i=1, \cdots, d$.
\end{lemma}
\begin{proof} Suppose that there exist two vertices $u, v$ such that there are $1\le p\le d-1$ and $1\le q\le d-1$  in $N(u)$ and $N(v)$, respectively. Then  $\widetilde{T}$ can be decomposed as  three subtrees $T_0$, $T_1$ and $T_2$, where $
T_0$ contains vertices $ u, v$;  $T_1$ has root $u$ and the branches $L_1=C_1$, $L_2\neq C_1$, $L_3, \cdots, L_d$; $T_2$ has root $v$ and the branches $R_1=C_1$, $R_2\neq C_1$, $R_3,  \cdots, R_d$. Clearly $\tau(S(L_1), x)>\tau(S(R_2), x)$. Hence by Corollary~\ref{exchange},
$$1>\min \{\tau(S(L_i),x),~1\le i\le d\}\ge \max\{\tau(S(R_i),x),~1\le i\le d\}=1,$$
which is a contradiction. Hence  there exists at most one vertex $u$ with $N(u)$ having $1\le k\le d-1$ pendent vertices.
Further suppose that there exists a vertex $u$ such that there are $1\le k\le d-1$ pendent vertices in $N(u)$.

  If the branches $T_1, \cdots, T_{d+1}$ of $T-u$ contains no $C_3$, it is easy to see that $T_i=C_1$ or $C_2$ for $i=1, \cdots, d+1$, which implies $\widetilde{T}$ is a greedy tree $T_{d+1}^*$.
If one of the branches $T_1, \cdots, T_{d+1}$ contains $C_3$, say $\widetilde{T}$ contains $C_3$, then there exists a vertex $w$ such that $\widetilde{T}$ can be decomposed as  trees $T_0$, $T_1$ and $T_2$, where $T_0$ contains vertices $ u, w$,  $T_1$ has root $u$ and the branches $L_1=C_1$, $L_2\neq C_1$, $L_3, \ldots, L_d$, $T_2$ has root $w$ and the branches $R_1=C_2$, $R_2$, $R_3,  \ldots, R_d$.  Clearly $\tau(S(L_1), x)>\tau(S(R_1), x)$. By Corollary~\ref{exchange},
$$\min \{\tau(S(L_i),x),~1\le i\le d\}\ge \max\{\tau(S(R_i),x),~1\le i\le d\}\ge \tau(S(C_2),x).$$
Hence $L_i=C_1$ or $C_2$ for $i=1, \cdots, d$.
\end{proof}

The proofs of the following two lemmas are similar to the proof of Lemma~\ref{degeled}, thus we omit the proof. The readers can refer to appendix.
\begin{lemma}\label{deg1ord+11}
Let $S(T)$ be an optimal tree  in $\mathcal{S(T)}_{n, d+1}$. If the degree of each vertex is 1 or $d+1$ in $V(T)$,  and   there is a vertex $v$ such that there are $1\le h\le d-1 $ pendent vertices in $N(v)$, then $T$ is greedy tree $T_{d+1}^*$.
\end{lemma}

\begin{lemma}\label{case3}
Let $S(T)$ be an optimal tree  in $\mathcal{S(T)}_{n, d+1}$.  If  the degree of each vertex is 1 or $d+1$ in $V(T)$,  and  there are  $d$ pendente vertices or no pendent vertices in $N(u)$ for any $u\in V(T)$. Then $T$ is greedy tree $T_{d+1}^*$.
\end{lemma}

\begin{theorem}\label{matching}
$S(T_{d+1}^*) $ is only optimal tree in $\mathcal{S(T)}_{n, d+1}$. In other words,
If $T$ is any tree of order $n$ with maximum degree $d+1$, then
$$M(S(T_{d+1}^*),x)\le M(S(T), x), \ \mbox{for }\ x>0, $$
with equality if and only if $T=T_{d+1}^*$.
\end{theorem}
\begin{proof} Let $T$ be any tree of order $n$ with maximum degree $d+1$.  If there exists a vertex with degree less than $d+1$, then by Corollary~\ref{cor27}, Lemmas~\ref{exchange-cor-d} and \ref{degeled}, $M(S(T_{d+1}^*,x)\le M(S(T), x), \ \mbox{for }\ x>0 $ with equality if and only if $T=T_{d+1}^*$.
If the degree of each vertex in $V(T)$  is $1$ or $d+1$,  and there exists a vertex $u$ such that there $1\le k\le d-1$ pendent vertices in $N(u)$, then by Lemmas~\ref{d+1} and \ref{deg1ord+11},  $M(S(T_{d+1}^*,x)\le M(S(T), x), \ \mbox{for }\ x>0 $ with equality if and only if $T=T_{d+1}^*$.
Hence we assume that degree of each vertex $v$ in $V(T)$  is $1$ or $d+1$, and  there are $d$ pendent vertices or no pendent vertices in $N(v)$, then by Lemma~\ref{case3}, $M(S(T_{d+1}^*,x)\le M(S(T), x), \ \mbox{for }\ x>0 $ with equality if and only if $T=T_{d+1}^*$. Therefore the assertion holds.
\end{proof}

\begin{lemma}\label{zhougutman}\cite{zhou 2008}
For every tree $T$ of order $n,$
$$
c_{k}(T ) = m_k(S(T )), k=0, \cdots, n.$$
Moreover,
$\varphi(T, x)=M(S(T), x).$
\end{lemma}
Now we are ready to prove Theorem~\ref{main1}.

\begin{proof} of  Theorem~\ref{main1}. It follows from Theorem~\ref{matching} and Lemma~\ref{zhougutman} that the assertion holds.
\end{proof}
Hosoya in 1971 \cite{hosoya1971} introduced a molecular graph structure descriptor $Z(T)$, which is now called the {\it Hosoya index,}
 $Z(T)=\sum_{k\ge 0} m(T, k).$
 Wager and Gutman  \cite{wagner2010} surveyed  properties and techniques for the Hosoya index.  We present a result for the Hosoya index.
 \begin{corollary}\cite{wagner2010}
 Let $T$ be any tree of order $n$ with the maximum degree $d+1$. Then
 $Z(S(T))\ge Z(S(T_{d+1}^*))$ with equality if and only if $T=T_{d+1}^*$.
 \end{corollary}
 \begin{proof}  It follows from Theorem~\ref{matching} with $x=1$.
 \end{proof}
\section{Proof of Theorem~\ref{main2} }
In order to prove  Theorem~\ref{main2}, we need more notions and Lemmas.
{\it The Laplacian-like energy} \cite{Liu2008} of  a tree $T$, $LEL$ for short, is defined as
$LEL(T)=\sum_{k=1}^{n-1} \sqrt{\mu_{k}},$
where $\mu_1\ge \mu_2\ge \cdots \ge \mu_n=0$ are the eigenvalues of $L(G)$.
 The {\it characteristic polynomial} of a tree $T$ is
$\mathcal{A}(T, \lambda)=\det(\lambda I-A(T))=\sum_{k=0}^n(-1)^ka_k\lambda^{n-k}.$
If  $\lambda_1(T)\ge\lambda_2(T)\ge\cdots\ge\lambda_n(T)$ are the eigenvalues of $A(T)$, then the {\it energy} of $T$ is
$E(T)=\sum_{k=1}^n|\lambda_k(T)|.$
Moreover, the {\it matching polynomial} of  $T$ is defined to be
 $\mathcal{M}(T, \lambda)=\sum_{k=0}^{\lfloor \frac{n}{2}\rfloor}(-1)^k m(T, k)\lambda^{n-2k}.$

\begin{lemma}\label{energy}
Let $T$ be any tree of order $n$. Then
\begin{equation}
IE(T)=LEL(T)=\frac{1}{2}E(S(T)).
\end{equation}
\end{lemma}
\begin{proof}
By the definition of $IE(T)$, $IE(T)$ is the sum of the square roots of all eigenvalues of  $I(G)I(G)^T$. Note that $I(G)I(G)^T=Q(T)$, which is signless Laplacian matrix. Since $T$ is  bipartite, $L(T)$ and $Q(T)$ are similar, which implies they have the same eigenvalues. Hence
$IE(T)=LEL(T)$. On the other hand, it follows from \cite{Yan2006} that
 $$\mathcal{L}(T, \lambda^2)=\lambda \mathcal{A}(S(T), \lambda).$$
 Then
 the adjacency eigenvalues of $S(T)$ are $\pm\sqrt{\mu_1(T)}, \cdots, \pm\sqrt{\mu_{n-1}(T)}, 0$, where $\mu_1(T), \cdots, \mu_{n-1}(T), 0$ are all eigenvalues of $L(T)$.  Hence
 $$E(S(T))=2\sum_{i=1}^{n-1}\sqrt{\mu_i(T)}=2LEL(T).$$
So the assertion holds.
\end{proof}
Now we are ready to prove Theorem~\ref{main2}.

{\bf Proof} Let $T$ be any tree of order $n$ with the maximum degree $d+1$. By the Coulson integral formula for energy (for example, see \cite{heuberger2009b}),
$$E(S(T))=\frac{2}{\pi}\int_0^\infty x^{-2}\log \left(\sum_{k\ge 0} m(S(T), k)x^{2k}\right) dx.$$
On the other hand, by Theorem~\ref{matching}, $M(S(T), x)\ge M(S(T_{d+1}^*), x)$ for $x>0$ with equality if and only if $T=T_{d+1}^*$. Hence
$ \sum_{k\ge 0} m(S(T), k)x^{2k}\ge \sum_{k\ge 0} m(S(T_{d+1}^*), k)x^{2k}$, which implies $E(S(T))\ge E(S(T_{d+1}^*))$ with equality if and only if $T=T_{d+1}^*$. Therefore, the assertion follows from Lemma~\ref{energy}.
$\blacksquare.$

Denote by $B_{n,d+1}$ the tree obtained by identifying the center vertex of the star $S_{d+1}$ and one of the pendent vertices of the path $P_{n-d}$.
\begin{lemma}\label{transform}\cite{Ilic2009}
Let $T\in \mathcal{T}_{n,d+1}$, then $c_k(T)\le c_k(B_{n,d+1}), ~k=0, 1, 2,...,n$.
\end{lemma}
 Combining with Theorems~\ref{main1} and \ref{matching} and Lemma~\ref{transform}, we are able to get the following results.
\begin{theorem}\label{maximum}
Let $T$ be any tree of order $n$ with the maximum degree $d+1$. Then for $x>0$,
$$M(S(T_{d+1}^*), x)\le M(S(T),x)\le  M(S(B_{n,d+1}),x),$$
$$\varphi(T_{d+1}^*, x)\le \varphi(T, x)\le\varphi(B_{n,d+1},x)
$$
 with left (right) equality if and only if $T=T_{d+1}^*$ ($T=B_{n,d+1})$.
\end{theorem}

\begin{lemma}\label{d-d+1}
 Let $S(T^*_{d}),~S(T^*_{d+1})$ be the optimal trees in $\mathcal{S(T)}_{n, d}$ and $\mathcal{S(T)}_{n, d+1}$ respectively. Then $M(S(T^*_{d}),x)> M(S(T^*_{d+1},x)$ for $x>0$ and  $n>d+1$.
\end{lemma}
\begin{proof}
Since $n>d+1$, then we can find two non-pendent vertices $u,~v$ in $S(T^*_{d})$. By Corollary~\ref{exchange}, there is at most one vertex $w$ such that $2\le deg(w)\le d$. Use Lemma~\ref{exchange0} for vertices $u,~v$ in $S(T^*_{d})$, then we can get a new graph $T'\in \mathcal{S(T)}_{n, d+1}$ such that $M(S(T^*_{d}),x)> M(S(T',x)\ge M(S(T^*_{d+1},x)$. This completes the proof.
\end{proof}
Hence it is easy to see the following assertion holds.
\begin{corollary} (\cite{zhou 2008},\cite{Mohar2007})
Let $T$ be any tree of order $n$. Then
$$\varphi(K_{1, n-1}, x)\le \varphi(T, x)\le \varphi(P_n, x),$$
$$M(S(K_{1,n}),x)\le M(S(T), x)\le  M(S(P_{n},x)$$
 for $x>0$ with left (right) equality if and only if $T=K_{1,n-1}$ ($T=P_n  $).
\end{corollary}
Based on the above results, we conclude this paper with the following conjecture.
\begin{conjecture} Let $T$ be any  tree of order $n$ with the maximum degree $d+1$. Then
$T_{d+1}^*\preceq T$ with equality if and only if $T=T_{d+1}^*$, i.e., $c_k(T_{d+1}^*)\le c_k(T)$ for $k=0, \cdots, n$ with all equalities if and only if $T=T_{d+1}^*$.
\end{conjecture}

\newpage

{\huge{{\bf Appendix}}}

In here, we present detail proof of Lemmas~\ref{deg1ord+11} and \ref{case3}

{\bf Lemma~\ref{deg1ord+11}}
Let $S(T)$ be an optimal tree  in $\mathcal{S(T)}_{n, d+1}$. If  degree of each vertex is 1 or $d+1$ in $V(T)$,  and   there is a vertex $v$ such that there are $1\le h\le d-1 $ pendent vertices in $N(v)$, then $T$ is greedy tree $T_{d+1}^*$.

\begin{proof}  If the branches $T_1, \cdots, T_{d+1}$ of $T-u$ contains no $C_3$, then by Lemma~\ref{d+1}, $T$ is a greedy tree. Now assume that the branches $T_1, \cdots, T_{d+1}$ of $T-u$ contains $C_3$, say $T_{d+1}\supseteq C_3$, then  by Lemma~\ref{d+1}, $T_1=C_2$ and $T_i=C_1$ or $C_2$ for $i=2, \cdots, d$. Let
 $P$ (see Fig.~$1$) be the longest path which goes through $v$ and terminates at non-pendent vertices with $v_2=v$.  Assume the length of $P$ is $2k$.
Let $L_{1}^i,...,L_{d}^i$ be the branches of $T-v_i$ containing no $w_1$ and $L_{1}^i$ contains $v_2$ for $i=3, \cdots, k$. Let $R_{1}^i,...,R_{d}^i$ be the branches of $T-w_i$ containing no $v_1$, and $R_{1}^i$ contains $w_1$ for $2\le i\le k$.  Clearly $R_{1}^2=R_{2}^2=\cdots=R_{d}^2=C_2$.

{\bf Claim : }
\begin{itemize} \setlength{\itemsep}{-\itemsep}
\item (1). $\tau(S(C_t),x)<\tau(S(L_{1}^t),x)< \tau(S(C_{t-1}),x),~ t=3, \cdots,  k+1$.
\item  (2). $L_{i}^t= C_{t-1}$ or $~C_{t} ,~ t=3, \cdots,  k+1,~ i=2, \cdots, d$.
\item (3). $R_{1}^t=R_{2}^t=\cdots=R_{d}^t=C_t,~ t=3, \cdots, k$.
\end{itemize}

We prove Claim by the induction on $t$.
For $t=3$, by Lemma~\ref{d+1},
 \begin{eqnarray*}
 \tau(S(C_t),x)&=& \frac{1}{1+\sum_{j=1}^d \frac{x}{1+x\tau(S(C_2),x)}}          \\
  &<& \frac{1}{1+\sum_{j=1}^d \frac{x}{1+x\tau(S(L_j^2),x)}}=\tau(S(L_1^3),x)\\
  &<& \frac{1}{1+\sum_{j=1}^d \frac{x}{1+x\tau(S(C_1),x)}} =\tau(S(C_2),x).
  \end{eqnarray*}
    So  Claim (1) holds for $t=3$.
By Lemma~\ref{d+1}, Claim (2) holds for $t=3$. Moreover,  by $\tau(S(L_1^3),x)>\tau(S(C_3),x)=\tau(S(R_1^3),x) $ and Corollary~\ref{exchange}, $\tau(S(L_1^3),x)\ge \max\{\tau(S(R_1^3),x), \cdots, \tau(S(R_d^3),x)$.  Combining with Lemma~\ref{d+1} and $\tau(S(C_2),x)>\tau(S(L_1^3),x)$, we have  $R_i^3=C_3$ for $i=2, \cdots, d$. Therefore
Claim (3) holds for $t=3$.

Assume that Claim holds for less than $t$ and we consider Claim for $t$. By the induction prothesis,
$$\tau(S(C_{t-1}),x)\le \tau(S(L_{1}^{t-1}),x)\le \tau(S(C_{t-2}),x)$$
and $\tau(S(C_{t-1}),x)\le \tau(S(L_{j}^{t-1}),x)\le \tau(S(C_{t-2}),x)$  for $j=2, \cdots, d$. It follows from   $(\ref{LM-tau})$ and (\ref{ch-ch-1}) that
\begin{eqnarray*}
\tau(S(C_{t}),x)&=&
\frac{1}{1+\sum_{i=1}^d\frac{x}{1+x\tau(S(C_{t-1}),x)}}\\
&<& \frac{1}{1+\sum_{i=1}^d\frac{x}{1+x\tau(S(L_{i}^{t-1}),x)}}\\
&=&\tau(S(L_{1}^t),x)\\
&<&\frac{1}{1+\sum_{i=1}^d\frac{x}{1+x\tau(S(C_{t-2}),x)}}.\\
&=&\tau(S(C_{t-1}),x)
\end{eqnarray*}
 Hence (1) holds for $t$.
 In order to prove (2) holds for $t$, we first prove the following several Claims

 {\bf Claim~2.1:} $\tau(S(C_t), x)\le \tau(S(L_i^t), x)\le \tau(S(C_{t-1}, x),$ for $i=2, \cdots, d$.

  In fact, there are $d$ branches $L_1^t, \cdots, L_d^t $ containing no $w_t$ in $T_{d+1}^*-v_t$ and there are $d$ branches $C_t, R_2^t,\cdots, R_d^t$ containing no $v_t$ in
     $T_{d+1}^*-w_t$. By  Claim~(1),  we have $\tau(S(C_t),x)<\tau(S(L_{1}^t),x)$. Hence by Corollary~\ref{exchange}, $\min\{\tau(S(L_1^t), x), \cdots, \tau(S(L_d^t), x)\}\ge
\tau(S(C_t),x)$. On the other hand,  there are $d$ branches $L_1^t, \cdots, L_d^t $ containing no $w_{t-1}$ in $T_{d+1}^*-v_t$ and there are $d$ branches $C_{t-1}, \cdots, C_{t-1}$ containing no $v_t$ in
     $T_{d+1}^*-w_{t-1}$. By Claim~(1), $\tau(S(L_{1}^t),x)<\tau(S(C_{t-1}),x).$ Hence by Corollary~\ref{exchange}, $\max\{S(L_{2}^t),x), \cdots, S(L_{d}^t),x)\}\le \tau(S(C_{t-1}),x).$ So Claim~2.1 holds.

 Let the maximum distance  between $v_t$ and any vertex in $L_1^t, \cdots, L_d^t$ is $l+1$.  Denote by
 $$V_j=\{u\ |\  dist(u, v_t)=l-j+1, u\in \bigcup_{i=1}^d V(L_i^t)\}, j=0, \cdots, l.$$

     {\bf Claim 2.2:} For any $u\in V_{l-j}$, there are $d$ the branches
     $L_1^u, \cdots, L_d^u$ containing no $v_t$ in $T_{d+1}^*-u$ such that
     \begin{equation}
     \tau(S(C_{t-j-1}), x)\le \tau(S(L_i^u), x) \le \tau(S(C_{t-j-2}), x)
     , \ \  i=1, \cdots, d,
     \end{equation} where $j=0, \cdots, \min\{t, l\}-2$.

We prove Claim~2.2 by the induction on $j$.  Let  $L_1^u, \cdots, L_d^u$  be $d$ the branches
    containing no $v_{t}$ in $T_{d+1}^*-u$ and  $T^u$ be the subtree  consisting of $u$ and $L_1^u, \cdots, L_d^u$. For $j=0$,
there exists a $1\le p\le d$ such that $T^u=L_p^t$.   If there exists an $1\le i\le d$ such that  $ \tau(S(L_i^u), x) < \tau(S(C_{t-1}), x)$, let    $R_1^{t-1}=C_{t-1}, \cdots, R_d^{t-1}$  be $d$ the branches
    containing no $v_t$ in $T_{d+1}^*-w_{t-1}$. Hence by Corollary~\ref{exchange},
    $\max\{  \tau(S(L_1^u), x), \cdots,  \tau(S(L_d^u), x)\}\le  \tau(S(C_{t-1}), x)$.
    Then
    $$ \tau(S(L_p^{t}), x)=\tau(S(T^u), x)=\frac{1}{1+\sum_{q=1}^d\frac{x}{1+x\tau(S(L_q^u), x)}}
    <\frac{1}{1+\sum_{q=1}^d\frac{x}{1+x\tau(S(C_{t-1}), x)}}= \tau(S(C_{t}), x),$$
     which contradicts Claim~2.1. Therefore,
     $$ \tau(S(C_{t-1}), x)\le \tau(S(L_i^u), x), i=1, \cdots, d.$$
     On the other hand, if there exists $1\le i\le d$ such that
     $ \tau(S(L_i^u), x) > \tau(S(C_{t-2}), x)$.   let    $R_1^{t-2}=C_{t-2}, \cdots, R_d^{t-2}$  be $d$ the branches
    containing no $v_t$ in $T_{d+1}^*-w_{t-2}$. By
    Corollary~\ref{exchange},
    $$\min\{  \tau(S(L_1^u), x), \cdots,  \tau(S(L_d^u), x)\}\le  \tau(S(C_{t-2}), x).$$
    Then
      $$ \tau(S(L_p^{t}), x)=\tau(S(T^u), x)=\frac{1}{1+\sum_{q=1}^d\frac{x}{1+x\tau(S(L_q^u), x)}}
    >\frac{1}{1+\sum_{q=1}^d\frac{x}{1+x\tau(S(C_{t-1}), x)}}= \tau(S(C_{t-1}), x),$$
     which contradicts Claim~2.1. Hence Claim~2.2 holds for $j=0$.
     Now assume that Claim~2.2 holds for $j$ and consider the claim for $j+1$.
     For any $u\in V_{l-(j+1)}$,  let
      $L_1^u, \cdots, L_d^u$  be $d$ the branches containing no $v_t$ in $T_{d+1}^*-u$  and  $T^u$ be the subtree  consisting of $u$ and $L_1^u, \cdots, L_d^u$.
Clearly there exists a  $u^{\prime}\in V_{l-j}$ such that there exists a branch $L_1^{u^{\prime}}$ in $T_{d+1}^*-u^{\prime}$   such that $T^u=L_1^{u^{\prime}}$.

     If there exists an $1\le i\le d$ such that  $ \tau(S(L_i^u), x) < \tau(S(C_{t-j-2}), x)$,    let $R_1^{t-j-2}=C_{t-j-2}, \cdots, R_d^{t-j-2}$  be $d$ the branches
    containing no $u$ in $T_{d+1}^*-w_{t-j-1}$. By Corollary~\ref{exchange},
    $$\max\{  \tau(S(L_1^u), x), \cdots,  \tau(S(L_d^u), x)\}\le  \tau(S(C_{t-j-2}), x).$$
          Then
    $$ \tau(S(L_1^{u^{\prime}}), x)=\tau(S(T^{u}), x)=\frac{1}{1+\sum_{q=1}^d\frac{x}{1+x\tau(S(L_q^u), x)}}
    <\frac{1}{1+\sum_{q=1}^d\frac{x}{1+x\tau(S(C_{t-j-2}), x)}}= \tau(S(C_{t-j-1}), x),$$
     which contradicts the induction hypothesis. Therefore, for any $u\in V_{l-j-1}$,
     $$ \tau(S(C_{t-j-2}), x)\le \tau(S(L_i^u), x), i=1, \cdots, d.$$
     On the other hand, if there exists $1\le i\le d$ such that
     $ \tau(S(L_i^u), x) > \tau(S(C_{t-j-3}), x)$.   Let    $R_1^{t-j-3}=C_{t-j-3}, \cdots, R_d^{t-j-3}$  be $d$ the branches
    containing no $u$ in $T_{d+1}^*-w_{t-j-2}$. By Corollary~\ref{exchange},
    $$\min\{  \tau(S(L_1^u), x), \cdots,  \tau(S(L_d^u), x)\}\ge  \tau(S(C_{t-j-3}), x).$$
    Then
      $$ \tau(S(L_1^{u^{\prime}}), x)=\tau(S(T^u), x)=\frac{1}{1+\sum_{q=1}^d\frac{x}{1+x\tau(S(L_q^u), x)}}
    >\frac{1}{1+\sum_{q=1}^d\frac{x}{1+x\tau(S(C_{t-j-3}), x)}}= \tau(S(C_{t-j-2}), x),$$
     which contradicts the induction hypothesis. Hence Claim~2.2 holds for $j+1$.
          Therefore Claim~2.2 holds.

   {\bf Claim~2.3:} $l=t-1$.

   If $l>t-1$, by Claim 2.2, for any $u\in V_{l-t+2},$
   $$  \tau(S(C_{1}), x) \le \tau(S(L_i^u),x), \ i=1, \cdots, d.$$
   On the other hand, there exists a $u^{\prime}\in V_{l-t+2}$ such that
   the largest distance between $u^{\prime}$ and the pendent vertex  is at least 2, then $C_2$ is a proper subgraph $L_1^{u^{\prime}}$, which implies   $\tau(S(L_1^{u^{\prime}}), x)\le  \tau(S(C_{2}), x)$. it is a contradiction.
   Hence $l\le t-1$.  Since $l\ge t-1$, then $l=t-1$.

    {\bf Claim 2.4:} For any $u\in V_{t-j-1}, ~j=0, \cdots, t-3$. Let $L_1^u, \cdots, \cdots, L_d^u$ be the $d$ branches containing no $v_t$ in $T_{d+1}^*-u$ and $T^u$ consist of $u$ and $d$ branches  $L_1^u, \cdots, \cdots, L_d^u$. Then
    $L_1^u=\cdots=L_d^u=C_{t-j-1}$ or $L_1^u=\cdots=L_d^u=C_{t-j-2}$, i.e.,  $T^u=C_{t-j}$ or $T^u=C_{t-j-1}$.

     We use induction for $t-j-1$. In fact, for $j=t-3$ and $u\in V_{2},$  by Claim 2.2, $\tau(S(C_{2}), x)\le \tau(S(L_i^u), x)\le \tau(S(C_{1}), x) $ for $ i=1, \cdots, d$. Hence
    $L_i^u=C_2$ or $L_i^u=C_1$ for $ i=1, \cdots, d$. If, say $L_1^u=C_2$ and $L_2^u=C_1$, then by  $\tau(S(L_1^2), x)>\tau(S(L_1^u),x)$ and Corollary~\ref{exchange},
    $\tau(S(L_1^2), x)\ge \max\{\tau(S(L_1^u), x), \cdots, \tau(S(L_d^u), x)\}\ge \tau(S(C_1), x)$, which is a contradiction. Hence $L_1^u=\cdots=L_d^u=C_{2}$ or
    $L_1^u=\cdots=L_d^u=C_{1}$, i.e.,  $T^u=C_{3}$ or $T^u=C_{2}$ for $u\in V_2$.
    Assume that Claim 2.4 hold for any vertex in $ V_{t-j-2}$. Now for $u\in   V_{t-j-1}$.
    Let $z_1, \cdots, z_d\in V_{t-j-1}$ be the roots of $L_1^u, \cdots, L_d^u$,
     respectively. By the induction hypothesis, $L_1^u, \cdots, L_d^u\in \{C_{t-j-1}, C_{t-j-2}\}$. Further $L_1^u=\cdots= L_d^u= C_{t-j-1}$ or $L_1^u=\cdots= L_d^u= C_{t-j-2}$. In fact, if, say  $L_1^u=C_{t-j-1}$ and $ L_2^u=C_{t-j-2}$, By
     $\tau(S(L_1^{t-j-1}),x)>\tau(S(C_{t-j-1}),x)$ and Corollary~\ref{exchange},
     $$\tau(S(L_1^{t-j-1}),x)\ge \max\{\tau(S(C_{t-j-1}),x), \tau(S(C_{t-j-2}),x)\ge \tau(S(C_{t-j-2}),x),$$
    which  contradiction to Claim 2.1.  Hence $L_1^u=\cdots= L_d^u= C_{t-j-1}$ or $L_1^u=\cdots= L_d^u= C_{t-j-2}$, i.e., $T^u=C_{t-j}$ or $T^u=C_{t-j-1}$.
So Claim 2.4 holds. Hence $L_i^t=C_t$ or $L_i=C_{t-1}$ for $i=2, \cdots, d$. In other words, Claim (2) holds.

Similarly, we can prove Claim (3) and omit the detail.
It is easy from Claim that $T_{d+1}^*$ is greedy tree. If the length of $P$ is odd, using similar way to prove this assertion. So we finish our proof.
\end{proof}

{\bf Lemma~\ref{case3}}
Let $S(T)$ be an optimal tree  in $\mathcal{S(T)}_{n, d+1}$.  If  degree of each vertex is 1 or $d+1$ in $V(T)$,  and  there are  $d$ pendente vertices or no pendent vertices in $N(u)$ for any $u\in V(T)$. Then $T$ is greedy tree $T_{d+1}^*$.

\begin{proof}

 Let $U=\{ u\ |\ deg (u)=1, u\in V(T)\}$ and $dist(v, U)=\min\{dist (v, u), u\in U\}$.
 $$U_i=\{v\ |\  dist(v, U)=i\},~i=1,2,... .$$
Clearly, for any $v\in U_{1},$ $T-v$ has $d$ branches $C_1, \cdots, C_1$.  If for any vertex $v\in U_{i},$ $T-v$ has $d$ branches $C_i, \cdots, C_i$ for $i=1, 2, \cdots, $ then $T$ is the greedy tree $T_{d+1}^*$, and we complete the proof. Next assume that there exists a vertex $v\in U_{i}$ such that $T-v$ has at least two branches different from $C_{i}$. Let $t_0$ be the smallest integer such that for any vertex $v\in U_{i},$ $T-v$ has $d$ branches $C_i, \cdots, C_i$ for $i=1, \cdots, t_0-1$ and for some vertex $v\in U_{t_0}$, $T-v$ has at least two branches different from $C_{t_0}$.  Let $P$ be the longest path through $v=v_{s}\in U_{t_0}$ (see Fig~2). Then by $v\in V_{t_0}$, we have $s\ge t_0+1$.  Assume that the length of $P$ is even.

Let $L_{1}^i,...,L_{d}^i$ be the branches of $T-v_i$ containing no $w_1$ and $L_{1}^i$ contains $v_1$ for $i=2, \cdots, k$. Let $R_{1}^i,...,R_{d}^i$ be the branches of $T-w_i$ containing no $v_1$, and $R_{1}^i$ contains $w_1$ for $2\le i\le k$.
By the definition of $v=v_{s}$, $C_{t_0}$ is a branch of $T-v_{s}$ and $ C_{t_0+1}$ is a  subtree of $L_j^{s}$ if $L_i^{s}\neq C_{t_0},~i=1,2,...,d$. Then  $\tau(S(L_j^{s}),x)\le \tau(S(C_{t_0+1}),x)$ for $L_i^{s}\neq C_{t_0}$, $i=1, \cdots. d$. Similarly, $C_{t_0}$ is a branch of $T-w_{t_0}$ and $ C_{t_0+1}$ is a  subtree of $R_j^{t_0}$ if $R_i^{t_0}\neq C_{t_0},~i=1,2,...,d$. Then  $\tau(S(R_j^{t_0}),x)\le \tau(S(C_{t_0+1}),x)$ for $R_i^{t_0}\neq C_{t_0}$, $i=1, \cdots. d$. If there is a branch of $R_1^{t_0},R_2^{t_0},\cdots,R_d^{t_0}$ is not $ C_{t_0}$, say $R_1^{t_0}\neq C_{t_0}$. By Corollary~\ref{exchange}, we have
\begin{eqnarray*}
\tau(S(R_1^{t_0}), x)&\ge& \min\{ \tau(S(R_1^{t_0}), x), \cdots,  \tau(S(R_d^{t_0}), x)\}  \\
&\ge&  \max\{  \tau(S(L_1^s), x), \cdots,  \tau(S(L_d^s), x)\}\ge \tau(S(C_{t_0}),x).
\end{eqnarray*}
It contradicts $\tau(S(R_1^{t_0}),x)\le \tau(S(C_{t_0+1}),x)$. Thus $R_1^{t_0}=\cdots=R_d^{t_0}=C_{t_0}$. Which will imply that $R_1^{t_0+1}=C_{t_0+1}$. Since $T-v_i$ has branches $L_{1}^s,...,L_{d}^s$ containing no $w_1$ and $T-w_{t_0+1}$ has branches $R_{1}^{t_0+1},...,R_{d}^{t_0+1}$ containing no $v_1$, by $\tau(S(C_{t_0+1}),x)<\tau(S(C_{t_0}),x)$ and Corollary~\ref{exchange}, we have
\begin{eqnarray*}
\tau(S(C_{t_0+1}),x)&\le& \max\{ \tau(S(R_1^{t_0+1}), x), \cdots,  \tau(S(R_d^{t_0+1}), x)\}  \\
&\le&  \min\{  \tau(S(L_1^s), x), \cdots,  \tau(S(L_d^s), x)\}\le \tau(S(C_{t_0+1}),x).
\end{eqnarray*}
Since $ C_{t_0+1}$ is a subtree of $L_i^{s}$ if $L_i^{s}\neq C_{t_0},~i=1,2,...,d$, then $L_i^{s}=C_{t_0+1}$ if $L_i^{s}\neq C_{t_0},~i=1,2,...,d$. This implies $s=t_0+1$. Without loss of generality, we can assume $L_1^{t_0+1}=C_{t_0+1}$ and $L_2^{t_0+1}=C_{t_0}$.

{\bf Claim : }
\begin{itemize} \setlength{\itemsep}{-\itemsep}
\item (1). $\tau(S(C_t),x)<\tau(S(L_{1}^t),x)< \tau(S(C_{t-1}),x),~ t=t_0+2, \cdots,  k+1$,

           $\tau(S(L_{i}^{t_0+1}),x)=\tau(S(C_{t_0+1}),x) ~\mbox{or}~\tau(S(C_{t_0}),x),~i=1, \cdots, d$.
\item  (2). $L_{i}^t= C_{t-1}$ or $~C_{t} ,~ t=t_0+1, \cdots,  k+1,~ i=2, \cdots, d$.
\item (3). $R_{1}^t=R_{2}^t=\cdots=R_{d}^t=C_t,~ t=t_0+1, \cdots, k$.
\end{itemize}

We prove Claim by the induction on $t$.
For $t=t_0+1$, by the above argument, we can find  Claim holds.
Assume that Claim holds for the number less than $t>t_0+1$ and we consider Claim for $t$.
 By the induction prothesis,
$$\tau(S(C_{t-1}),x)\le \tau(S(L_{1}^{t-1}),x)\le \tau(S(C_{t-2}),x)$$
and $\tau(S(C_{t-1}),x)\le \tau(S(L_{j}^{t-1}),x)\le \tau(S(C_{t-2}),x)$  for $j=2, \cdots, d$. It follows from   $(\ref{LM-tau})$ and (\ref{ch-ch-1}) that
\begin{eqnarray*}
\tau(S(C_{t}),x)&=&
\frac{1}{1+\sum_{i=1}^d\frac{x}{1+x\tau(S(C_{t-1}),x)}}\\
&<& \frac{1}{1+\sum_{i=1}^d\frac{x}{1+x\tau(S(L_{i}^{t-1}),x)}}\\
&=&\tau(S(L_{1}^t),x)\\
&<&\frac{1}{1+\sum_{i=1}^d\frac{x}{1+x\tau(S(C_{t-2}),x)}}.\\
&=&\tau(S(C_{t-1}),x)
\end{eqnarray*}
 Hence (1) holds for $t$.
 In order to prove (2) holds for $t$, we first prove the following several Claims.

 {\bf Claim~2.1:} $\tau(S(C_t), x)\le \tau(S(L_i^t), x)\le \tau(S(C_{t-1}, x)$ for $i=2, \cdots, d$.

  In fact, there are $d$ brances $L_1^t, \cdots, L_d^t $ containing no $w_t$ in $T_{d+1}^*-v_t$ and there are $d$ branches $R_1^t=C_t, R_2^t,\cdots, R_d^t$ containing no $v_t$ in
     $T_{d+1}^*-w_t$. By (1) of the Claim,  we have $\tau(S(C_t),x)<\tau(S(L_{1}^t),x)$. Hence by Corollary~\ref{exchange}, $\min\{\tau(S(L_1^t), x), \cdots, \tau(S(L_d^t), x)\}\ge
\tau(S(C_t),x)$. On the other hand,  there are $d$ branches $L_1^t, \cdots, L_d^t $ containing no $w_{t-1}$ in $T_{d+1}^*-v_t$ and there are $d$ branches $C_{t-1}, \cdots, C_{t-1}$ containing no $v_t$ in
     $T_{d+1}^*-w_{t-1}$. By (1) of the Claim, $\tau(S(L_{1}^t),x)<\tau(S(C_{t-1}),x).$ Hence by Corollary~\ref{exchange}, $\max\{S(L_{2}^t),x), \cdots, S(L_{d}^t),x)\}\le \tau(S(C_{t-1}),x).$ So Claim~2.1 holds.

 Let the maximum distance  between $v_t$ and any vertex in $L_1^t, \cdots, L_d^t$ is $l+1$. By the definition of $P$, we can find that $l\ge t-1$.  Denote by
 $$V_j=\{u\ |\  dist(u, v_t)=l-j+1, u\in \bigcup_{i=1}^d V(L_i^t)\}, j=0, \cdots, l.$$

     {\bf Claim 2.2:} For any $u\in V_{l-j}$, there are $d$ branches
     $L_1^u, \cdots, L_d^u$ containing no $v_t$ in $T_{d+1}^*-u$ such that
     \begin{equation}
     \tau(S(C_{t-j-1}), x)\le \tau(S(L_i^u), x) \le \tau(S(C_{t-j-2}), x)
     , \ \  i=1, \cdots, d,
     \end{equation} where $j=0, \cdots, \min\{t, l\}-2$.

     We prove Claim~2.2 by the induction on $j$.  Let  $L_1^u, \cdots, L_d^u$  be $d$ the branches
    containing no $v_{t}$ in $T_{d+1}^*-u$ and  $T^u$ be the subtree  consisting of $u$ and $L_1^u, \cdots, L_d^u$. For $j=0$,
there exists a $1\le p\le d$ such that $T^u=L_p^t$.   If there exists an $1\le i\le d$ such that  $ \tau(S(L_i^u), x) < \tau(S(C_{t-1}), x)$, let    $R_1^{t-1}=C_{t-1}, \cdots, R_d^{t-1}$  be $d$ the branches
    containing no $v_t$ in $T_{d+1}^*-w_{t-1}$. Hence by Corollary~\ref{exchange},
    $\max\{  \tau(S(L_1^u), x), \cdots,  \tau(S(L_d^u), x)\}\le  \tau(S(C_{t-1}), x)$.
    Then
    $$ \tau(S(L_p^{t}), x)=\tau(S(T^u), x)=\frac{1}{1+\sum_{q=1}^d\frac{x}{1+x\tau(S(L_q^u), x)}}
    <\frac{1}{1+\sum_{q=1}^d\frac{x}{1+x\tau(S(C_{t-1}), x)}}= \tau(S(C_{t}), x),$$
     which contradicts Claim~2.1. Therefore,
     $$ \tau(S(C_{t-1}), x)\le \tau(S(L_i^u), x), i=1, \cdots, d.$$
     On the other hand, if there exists $1\le i\le d$ such that
     $ \tau(S(L_i^u), x) > \tau(S(C_{t-2}), x)$.   let    $R_1^{t-2}=C_{t-2}, \cdots, R_d^{t-2}$  be $d$ the branches
    containing no $v_t$ in $T_{d+1}^*-w_{t-2}$. By
    Corollary~\ref{exchange},
    $$\min\{  \tau(S(L_1^u), x), \cdots,  \tau(S(L_d^u), x)\}\le  \tau(S(C_{t-2}), x).$$
    Then
      $$ \tau(S(L_p^{t}), x)=\tau(S(T^u), x)=\frac{1}{1+\sum_{q=1}^d\frac{x}{1+x\tau(S(L_q^u), x)}}
    >\frac{1}{1+\sum_{q=1}^d\frac{x}{1+x\tau(S(C_{t-1}), x)}}= \tau(S(C_{t-1}), x),$$
     which contradicts Claim~2.1. Hence Claim~2.2 holds for $j=0$.
     Now assume that Claim~2.2 holds for $j$ and consider the claim for $j+1$.
     For any $u\in V_{l-(j+1)}$,  let
      $L_1^u, \cdots, L_d^u$  be $d$ the branches containing no $v_t$ in $T_{d+1}^*-u$  and  $T^u$ be the subtree  consisting of $u$ and $L_1^u, \cdots, L_d^u$.
Clearly there exists a  $u^{\prime}\in V_{l-j}$ such that there exists a branch $L_1^{u^{\prime}}$ in $T_{d+1}^*-u^{\prime}$   such that $T^u=L_1^{u^{\prime}}$.

     If there exists an $1\le i\le d$ such that  $ \tau(S(L_i^u), x) < \tau(S(C_{t-j-2}), x)$,    let $R_1^{t-j-2}=C_{t-j-2}, \cdots, R_d^{t-j-2}$  be $d$ the branches
    containing no $u$ in $T_{d+1}^*-w_{t-j-1}$. By Corollary~\ref{exchange},
    $$\max\{  \tau(S(L_1^u), x), \cdots,  \tau(S(L_d^u), x)\}\le  \tau(S(C_{t-j-2}), x).$$
          Then
    $$ \tau(S(L_1^{u^{\prime}}), x)=\tau(S(T^{u}), x)=\frac{1}{1+\sum_{q=1}^d\frac{x}{1+x\tau(S(L_q^u), x)}}
    <\frac{1}{1+\sum_{q=1}^d\frac{x}{1+x\tau(S(C_{t-j-2}), x)}}= \tau(S(C_{t-j-1}), x),$$
     which contradicts the induction hypothesis. Therefore, for any $u\in V_{l-j-1}$,
     $$ \tau(S(C_{t-j-2}), x)\le \tau(S(L_i^u), x), i=1, \cdots, d.$$
     On the other hand, if there exists $1\le i\le d$ such that
     $ \tau(S(L_i^u), x) > \tau(S(C_{t-j-3}), x)$.   Let    $R_1^{t-j-3}=C_{t-j-3}, \cdots, R_d^{t-j-3}$  be $d$ the branches
    containing no $u$ in $T_{d+1}^*-w_{t-j-2}$. By Corollary~\ref{exchange},
    $$\min\{  \tau(S(L_1^u), x), \cdots,  \tau(S(L_d^u), x)\}\ge  \tau(S(C_{t-j-3}), x).$$
    Then
      $$ \tau(S(L_1^{u^{\prime}}), x)=\tau(S(T^u), x)=\frac{1}{1+\sum_{q=1}^d\frac{x}{1+x\tau(S(L_q^u), x)}}
    >\frac{1}{1+\sum_{q=1}^d\frac{x}{1+x\tau(S(C_{t-j-3}), x)}}= \tau(S(C_{t-j-2}), x),$$
     which contradicts the induction hypothesis. Hence Claim~2.2 holds for $j+1$.
          Therefore Claim~2.2 holds.

   {\bf Claim~2.3:} $l=t-1$.

   If $l>t-1$, by Claim 2.2, for any $u\in V_{l-t+2},$
   $$  \tau(S(C_{1}), x) \le \tau(S(L_i^u),x), \ i=1, \cdots, d.$$
   On the other hand, there exists a $u^{\prime}\in V_{l-t+2}$ such that
   the largest distance between $u^{\prime}$ and the pendent vertex  is at least 2, then $C_2$ is a proper subgraph $L_1^{u^{\prime}}$, which implies   $\tau(S(L_1^{u^{\prime}}), x)\le  \tau(S(C_{2}), x)$. it is a contradiction.
   Hence $l\le t-1$.  Since $l\ge t-1$, then $l=t-1$.

    {\bf Claim 2.4:} For any $u\in V_{t-j-1}, ~j=0, \cdots, t-3$. Let $L_1^u, \cdots, \cdots, L_d^u$ be the $d$ branches containing no $v_t$ in $T_{d+1}^*-u$ and $T^u$ consist of $u$ and $d$ branches  $L_1^u, \cdots, \cdots, L_d^u$. Then
    $L_1^u=\cdots=L_d^u=C_{t-j-1}$ or $L_1^u=\cdots=L_d^u=C_{t-j-2}$, i.e.,  $T^u=C_{t-j}$ or $T^u=C_{t-j-1}$.

     We use induction for $t-j-1$. In fact, for $j=t-3$ and $u\in V_{2},$  by Claim 2.2, $\tau(S(C_{2}), x)\le \tau(S(L_i^u), x)\le \tau(S(C_{1}), x) $ for $ i=1, \cdots, d$. Hence
    $L_i^u=C_2$ or $L_i^u=C_1$ for $ i=1, \cdots, d$. If, say $L_1^u=C_2$ and $L_2^u=C_1$, then by  $\tau(S(L_1^2), x)>\tau(S(L_1^u),x)$ and Corollary~\ref{exchange},
    $\tau(S(L_1^2), x)\ge \max\{\tau(S(L_1^u), x), \cdots, \tau(S(L_d^u), x)\}\ge \tau(S(C_1), x)$, which is a contradiction. Hence $L_1^u=\cdots=L_d^u=C_{2}$ or
    $L_1^u=\cdots=L_d^u=C_{1}$, i.e.,  $T^u=C_{3}$ or $T^u=C_{2}$ for $u\in V_2$.
    Assume that Claim 2.4 hold for any vertex in $ V_{t-j-2}$. Now for $u\in   V_{t-j-1}$.
    Let $z_1, \cdots, z_d\in V_{t-j-1}$ be the roots of $L_1^u, \cdots, L_d^u$,
     respectively. By the induction hypothesis, $L_1^u, \cdots, L_d^u\in \{C_{t-j-1}, C_{t-j-2}\}$. Further $L_1^u=\cdots= L_d^u= C_{t-j-1}$ or $L_1^u=\cdots= L_d^u= C_{t-j-2}$. In fact, if, say  $L_1^u=C_{t-j-1}$ and $ L_2^u=C_{t-j-2}$, By
     $\tau(S(L_1^{t-j-1}),x)>\tau(S(C_{t-j-1}),x)$ and Corollary~\ref{exchange},
     $$\tau(S(L_1^{t-j-1}),x)\ge \max\{\tau(S(C_{t-j-1}),x), \tau(S(C_{t-j-2}),x)\ge \tau(S(C_{t-j-2}),x),$$
    which  contradiction to Claim 2.1.  Hence $L_1^u=\cdots= L_d^u= C_{t-j-1}$ or $L_1^u=\cdots= L_d^u= C_{t-j-2}$, i.e., $T^u=C_{t-j}$ or $T^u=C_{t-j-1}$.
So Claim 2.4 holds.

Hence $L_i^t=C_t$ or $L_i^t=C_{t-1}$ for $i=2, \cdots, d$. In other words, (2) of Claim  holds.

Similarly, we can prove (3) of Claim, here we omit the detail.
It is easy from Claim that $T_{d+1}^*$ is greedy tree. If the length of $P$ is odd, using similar way to prove this assertion. So we finish our proof.
\end{proof}

\end{document}